\documentclass[11pt,a4 paper,oneside, reqno]{amsart}
\usepackage[utf8]{inputenc}
\usepackage[OT1]{fontenc}
\usepackage[english]{babel}
\usepackage{dsfont}
\usepackage{amsfonts}
\usepackage{amsmath}
\usepackage{bbm}
\usepackage{wasysym}
\usepackage{amsthm}
\usepackage{amssymb}
\usepackage{float}
\usepackage{textcomp}
\usepackage{xcolor}
\usepackage{graphicx}
\usepackage{fancybox}
\usepackage{fancyhdr}
\usepackage{mathrsfs}
\usepackage{tikz}
\usetikzlibrary{arrows}
\usetikzlibrary{calc}
\usepackage{centernot}
\usepackage{listings}
\usepackage{faktor}
\usepackage{bm}
\usepackage{setspace}
\usepackage[colorlinks,citecolor=blue,linkcolor=blue]{hyperref}
\usepackage{newlfont}
\usepackage{geometry}
\usepackage{ccaption}
\usepackage{tipa}
\usepackage{units}
\usepackage[autostyle,italian=guillemets]{csquotes}
\usepackage{comment}

\usepackage{guit}
\usepackage{tikz-cd}
\usepackage{enumerate}

\DeclareFontFamily{U}{mathx}{}
\DeclareFontShape{U}{mathx}{m}{n}{ <-> mathx10 }{}
\DeclareSymbolFont{mathx}{U}{mathx}{m}{n}
\DeclareFontSubstitution{U}{mathx}{m}{n}

\DeclareMathAccent{\widecheck}{0}{mathx}{"71}

\everymath{\displaystyle}

\theoremstyle{definition}
\newtheorem{Def}{Definition}[section]

\newtheorem{es}[Def]{Example}
\newtheorem{ese}[Def]{Examples}
\newtheorem{as}[Def]{Assumption}

\theoremstyle{remark}
\newtheorem{obs}[Def]{Remark}
\newtheorem{observation}[Def]{Observation}
\theoremstyle{plain}
\newtheorem{prop}[Def]{Proposition}

\newtheorem{lema}[Def]{Lemma}

\newtheorem{cor}[Def]{Corollary}

\newtheorem{teo}[Def]{Theorem}

\newcommand{\bo}{\mathbf}
\newcommand{\A}{{\mathcal A}}
\newcommand{\B}{{\mathcal B}}
\newcommand{\C}{{\mathcal C}}
\newcommand{\D}{{\mathcal D}}
\newcommand{\E}{{\mathcal E}}

\newcommand{\G}{{\mathcal G}}

\newcommand{\K}{{\mathcal K}}

\newcommand{\M}{{\mathcal M}}

\renewcommand{\P}{{\mathcal P}}
\newcommand{\Q}{{\mathcal Q}}

\renewcommand{\S}{{\mathcal S}}
\newcommand{\T}{{\mathcal T}}

\newcommand{\V}{{\mathcal V}}
\newcommand{\W}{{\mathcal W}}

\newcommand{\Z}{{\mathcal Z}}

\newcommand{\mt}{\mathfrak}
\newcommand{\tx}{\textnormal}

\newcommand{\op}{^\textnormal{op}}

\newcommand{\colim}{\operatornamewithlimits{colim}}

\newcommand{\psicheck}{\Psi^{\vee}}

\makeatletter
\newcommand{\changeoperator}[1]{%
	\csletcs{#1@saved}{#1@}%
	\csdef{#1@}{\changed@operator{#1}}%
}
\newcommand{\changed@operator}[1]{%
	\mathop{%
		\mathchoice{\textstyle\csuse{#1@saved}}
		{\csuse{#1@saved}}
		{\csuse{#1@saved}}
		{\csuse{#1@saved}}%
	}%
}
\makeatother

\makeatletter
\def\@tocline#1#2#3#4#5#6#7{\relax
	\ifnum #1>\c@tocdepth 
	\else
	\par \addpenalty\@secpenalty\addvspace{#2}%
	\begingroup \hyphenpenalty\@M
	\@ifempty{#4}{%
		\@tempdima\csname r@tocindent\number#1\endcsname\relax
	}{%
		\@tempdima#4\relax
	}%
	\parindent\z@ \leftskip#3\relax \advance\leftskip\@tempdima\relax
	\rightskip\@pnumwidth plus4em \parfillskip-\@pnumwidth
	#5\leavevmode\hskip-\@tempdima
	\ifcase #1
	\or\or \hskip 1em \or \hskip 2em \else \hskip 3em \fi%
	#6\nobreak\relax
	\hfill\hbox to\@pnumwidth{\@tocpagenum{#7}}\par
	\nobreak
	\endgroup
	\fi}
\makeatother

\changeoperator{sum}
\changeoperator{prod}
\changeoperator{coprod}

\title{Accessible categories with a class of limits}
\author{Stephen Lack and Giacomo Tendas}
\address{School of Mathematical and Physical Sciences, Macquarie University NSW 2109, 
	Australia}
\email{steve.lack@mq.edu.au}
\address{Department of Mathematics and Statistics, Masaryk University, Kotl\'{a}\v{r}sk\'{a} 2, 611 37 Brno, Czech Republic}
\email{tendasg@math.muni.cz}
\date{\today}
\thanks{The first-named author acknowledges with gratitude the support of an Australian Research Council Discovery Project DP190102432. The second-named author gratefully acknowledges the support of an International Macquarie University Research Excellence Scholarship 2018115/20191524 and of the Grant agency of the Czech republic under the grant 22-02964S}

\begin{document}
	
\begin{abstract}
	In this paper we characterize those accessible $\V$-categories that have limits of a specified class. We do this by introducing the notion of {companion} $\mt C$ for a class of weights $\Psi$, as a collection of special types of colimit diagrams that are compatible with $\Psi$. We then characterize the accessible $\V$-categories with $\Psi$-limits as those accessibly embedded and {\em $\mt C$-virtually reflective} in a presheaf $\V$-category, and as the $\V$-categories of {\em $\mt C$-models} of sketches. This allows us to recover the standard theorems for locally presentable, locally multipresentable, and locally polypresentable categories as instances of the same general framework. In addition, our theorem covers the case of any weakly sound class $\Psi$, and provides a new perspective on the case of weakly locally presentable categories.
\end{abstract}	
	
\maketitle
	
\tableofcontents

\section{Introduction}

An accessible category is complete if and only if it is cocomplete,
and in that case it is called locally presentable. Many of the
accessible categories of interest, however, are not complete, but have
only certain limits: perhaps just the products, or just the connected
limits --- or, in the case of 2-categories, just the flexible
limits. There are a number of results in the literature of the
following type: for a particular class $\Psi$ of limits, a category is
accessible with $\Psi$-limits if and only if it is accessible and
satisfies some ``relaxed'' cocompleteness condition, the details of which depend
on the class $\Psi$; furthermore, the accessible categories with
$\Psi$-limits can be characterized as the categories of models of a
particular type of sketch. The
goal of this paper is to find a general context, containing the known
cases, in which such a theorem can be proved.

In an earlier paper \cite{LT22:virtual}, we introduced an approach to
(enriched) accessibility based on what we called ``virtual'' notions,
in which, rather than work entirely in the given category $\A$, one
passes to the completion $\P^\dagger\A$ of $\A$ under limits. In many
cases this involves relaxing the requirement that a particular
functor defined on $\A$ be representable by the much weaker requirement that it be
small. Thus accessible categories need not be cocomplete, but they are
virtually cocomplete; likewise, they may not be reflective
subcategories of a presheaf category, but they are virtually
reflective.  In this paper, we develop a more nuanced version of this
approach: in order to study accessible categories with $\Psi$-limits,
we work with a suitably chosen full subcategory $\mt{C}^\dagger_1\A$
of $\P^\dagger\A$, depending on $\Psi$ but containing the
representables, and use it to define various {\em $\mt{C}$-virtual}
notions. In particular, we see that an accessible $\A$ is
$\Psi$-complete if and only if it is $\mt{C}$-virtually cocomplete,
and that it is then $\mt{C}$-virtually reflective in a presheaf
category. 

We begin by looking at the two original examples. The first is the locally
presentable case, involving genuine representability, with $\Psi$
consisting of all (small) limits. It goes all the way back to \cite{GU71:libro} for the
unenriched version, with the enriched version appearing in
\cite{Kel82:articolo}, and \cite{BQR98} for the sketch aspect. 

\begin{teo}
	The following are equivalent for a $\V$-category $\A$:\begin{enumerate}\setlength\itemsep{0.25em}
		\item $\A$ is accessible and complete;
		\item $\A$ is accessible and cocomplete;
		\item $\A$ is the $\V$-category of models of a limit sketch.
	\end{enumerate}
\end{teo}
\noindent In this example we have genuine cocompleteness, as well as genuine
reflectivity of the models in a presheaf category.

The second example, due to Diers~\cite{Die80:articolo}, involves the locally
multipresentable categories: these are the accessible categories with connected limits.

\begin{teo}
	The following are equivalent for a category $\A$:\begin{enumerate}\setlength\itemsep{0.25em}
		\item $\A$ is accessible with connected limits;
		\item $\A$ is accessible and multicocomplete;
		\item $\A$ is the category of models of a limit/coproduct sketch.
	\end{enumerate}
\end{teo}

How should we understand multicocompleteness? A category $\A$ is 
cocomplete if for each functor $S\colon\D\to\A$ with $\D$ small, the
functor $\A\to\bo{Set}$, sending $A\in\A$ to the set $[\D,\A](S,\Delta A)$
of cocones of $S$ with vertex $A$, is representable. It is
multicocomplete if instead each such functor $\A\to\bo{Set}$ is a (small)
coproduct of representables. Equivalently, $\A$ is multicocomplete
just when the free completion $\tx{Fam}^\dagger\A$ of $\A$ under
(small) products has colimits of diagrams lying in $\A$.

Thus we see that coproducts appear in both condition (2) and (3); the
connection with (1) is that coproducts commute in $\bo{Set}$ with
connected limits, and that general colimits can be constructed as
coproducts of connected colimits.

This situation can be generalized using the notion of weakly sound
class of weights \cite{ABLR02:articolo, LT22:virtual}. Here one starts
with a class $\Psi$ of weights, and considers the class $\Psi^+$ of
those weights whose colimits commute in $\V$ with $\Psi$-limits. 
Diers' theorem generalizes to this context:

\begin{teo}[see Theorem~\ref{sound-weakcharacterization}]\label{sound-intro}
	Let $\Psi$ be a weakly sound class and $\Phi=\Psi^+$ be the class of weights whose colimits commute with $\Psi$-limits in $\V$. The following are equivalent for a $\V$-category $\A$:\begin{enumerate}\setlength\itemsep{0.25em}
		\item $\A$ is accessible with $\Psi$-limits;
		\item $\A$ is accessible and $\Phi^\dagger \A$ has
                  colimits of diagrams in $\A$;
		\item $\A$ is the $\V$-category of models of a limit/$\Phi$-colimit sketch.
	\end{enumerate}
\end{teo}

This time $\Phi^\dagger\A$ is the free completion of $\A$ under
$\Phi$-weighted limits; thus if $\A$ is small then $\Phi^\dagger\A$ is the opposite of the
$\V$-category of all $\V$-functors $\A\to\V$ lying in $\Phi$.
Once again, the condition that $\Phi^\dagger\A$ have colimits of
diagrams in $\A$ can be understood as a relaxed representability
condition.
The colimit of $S\colon \D\to\A$ weighted by $F\colon \D\op\to\V$
exists just when the $\V$-functor $\A\to\V$ sending $A\in\A$ to
$[\D\op,\V](F,\A(S,A))$ is representable. The colimit exists in
$\Phi^\dagger\A$ just when the same  $\V$-functor $\A\to\V$ lies in
$\Phi^\dagger\A$: see Observation~\ref{virtual-col} below.

When $\Psi=\emptyset$, so that accessible $\V$-categories with
$\Psi$-limits are just accessible categories, $\Phi$ is the class $\P$
of all weights, and
we recover the characterization of accessible categories as the
sketchable ones \cite{Lai81:articolo,BQR98}, as well as the fact \cite[Remark~3.5]{DL07} that if $\A$
is accessible then $\P^\dagger\A$ has colimits of representables; in
other words, it  is virtually cocomplete in the sense of
\cite{LT22:virtual}.

When $\Psi=\P$, so that the accessible $\V$-categories with $\Psi$-limits
are the locally presentable $\V$-categories, $\Phi$ is the class $\Q$
of absolute/Cauchy weights. Since any accessible category is Cauchy
complete, our ``relaxed'' cocompleteness condition of $\Q^\dagger\A$
having colimits of diagrams in $\A$ just says that $\A$ is cocomplete,
and we recover the standard characterization of locally presentable
$\V$-categories.\footnote{Later, in
  Section~\ref{sound-companion}, we show how to work directly with the
  empty class $\Phi=\emptyset$ rather than $\Phi=\Q$: see in
  particular Remark~\ref{Rmk:why-not-Psi+}.}

Even though Theorem~\ref{sound-intro}  is already quite general and captures some
classes of limits that weren't considered before --- see
Example~\ref{soundcomp-examples} for several such --- it doesn't cover two important examples: the weakly locally presentable categories ($\Psi$ is the class for products) and the locally polypresentable categories ($\Psi$ is the class for wide pullbacks), the problem being that in those two cases the classes of limits in question are not weakly sound. 

In the case of locally polypresentable categories the known characterization theorem reads as follows:

\begin{teo}\label{poly-intro}
	The following are equivalent for a category $\A$:\begin{enumerate}\setlength\itemsep{0.25em}
		\item $\A$ is accessible with wide pullbacks;
		\item $\A$ is accessible and polycocomplete;
		\item $\A$ is the category of models of a galoisian sketch.
	\end{enumerate}
\end{teo}

The history behind the proof of this theorem is complicated: its origins are in Lamarche's doctoral thesis \cite{Lam89:PhD}, with further work by Taylor \cite{Tay90:articolo} and Hu--Tholen \cite{HT96qc:articolo}. The notion of galoisian sketch and the equivalence of (3) to the other conditions is due to Ageron \cite{Age92:articolo}. 

Unlike the previous cases, polycolimits in $\A$ are not computed in a
free completion of $\A$, at least not in the sense of a free
completion under limits with respect to a class of weights (or class
of diagram shapes). In fact, Hu and Tholen \cite{HT96qc:articolo}
prove that a category $\A$ has polycolimits if and only if the {\em
  free completion of $\A$ under limits of free groupoid actions} has
colimits of objects from $\A$. This is not a free completion in the
usual sense since the diagrams defining free groupoid actions are not
just functors out of an indexing category, but need to satisfy some
additional properties. Nonetheless, it is the
case that polycocompleteness does say that for any diagram
$S\colon\D\to\A$, the induced functor from $\A$ to $\bo{Set}$ sending $A$ to
$[\D,\A](S,\Delta A)$ lies in a suitably chosen full subcategory of
$\P^\dagger\A$. Colimits of free groupoid actions also play a role in
Ageron's notion of galoisian sketch \cite{Age92:articolo}.

To capture the locally polypresentable case and the sound case under the same theory we shall prove the following theorem, and introduce the various notions on which it relies.

\begin{teo}\label{companioin-intro}
	Let $\Psi$ be a class of weights, and $\mt C$ be a companion for $\Psi$. The following are equivalent for a $\V$-category $\A$:\begin{enumerate}\setlength\itemsep{0.25em}
		\item $\A$ is accessible with $\Psi$-limits;
		\item $\A$ is accessible and $\mt C^\dagger_1\A$ has colimits of representables.
		\item $\A$ is the $\V$-category of $\mt C$-models of a sketch.
	\end{enumerate}
\end{teo}

\noindent In particular we need to explain:\begin{itemize}\setlength\itemsep{0.25em}
	\item what it means for $\mt C$ to be a companion (Definition~\ref{companion});
	\item what $\mt C^\dagger_1\A$ is (Definition~\ref{C_1A});
	\item what is a $\mt C$-model of a sketch (Definition~\ref{C-models}).
\end{itemize}
For the moment, we merely point out that $\mt C^\dagger_1\A$ is a full
subcategory of $\P^\dagger\A$, and as such consists of certain (small)
$\V$-functors from $\A$ to $\V$. Thus, as we have seen, any notion
expressed in terms of representability of a $\V$-functor $\A\to\V$ can be relaxed so as
to ask only that the $\V$-functor lie in $\mt C^\dagger_1\A$. We refer
to this as the {\em $\mt C$-virtual} version of the notion; thus, for
example, condition (2) in the theorem says that $\A$ is $\mt
C$-virtually cocomplete.  Allowing $\Psi$ and $\mt C$ to vary, we obtain a whole spectrum of
versions of the ``virtual notions'' of \cite{LT22:virtual}.

If $\Psi$ is weakly sound, Theorem~\ref{companioin-intro} specializes to Theorem~\ref{sound-intro}. If $\V=\bo{Set}$ and $\Psi$ is the class for wide pullbacks we recover Theorem~\ref{poly-intro}. In the context of weakly locally presentable categories we do obtain a theorem but it does not exactly match the characterization theorems of \cite[Chapter~4]{AR94:libro} involving weak reflections and weak cocompleteness; those will be recovered separately in Section~\ref{wrsnc}.

In a follow-up paper \cite{Ten22:duality}, which like this paper is
based on material from the second author's PhD thesis \cite{Ten22:phd}, the notion of companion will be used to establish a duality between the 2-categories of finitely accessible categories with some class of limits and the 2-category of those lex categories which arise as free cocompletions of a small category under colimits of a specific type.

\subsection*{Outline}
We begin in Section~\ref{general-psi} by recalling some notation from
\cite{LT22:virtual} and introducing the general setting of the paper. Already in
Theorem~\ref{relative-quasiflattheorem} we prove a result of the same
flavour as those described above, but not yet sufficiently precise to
capture what is known in specific cases. The main results are discussed in
Section~\ref{main-psi} where we introduce the notion of companion and
prove the characterization theorems (\ref{C-strong-charact} and
\ref{relatice-C-charact}). In Section~\ref{sketches-C} we define the
notion of $\mt C$-model of a sketch whenever $\mt C$ is a companion
for a class $\Psi$, and give conditions on $\mt C$ under which the
accessible $\V$-categories with $\Psi$-limits coincide with the
$\V$-categories of $\mt C$-models  (Theorem~\ref{Psi-sketch}).

Section~\ref{Examples-companions} is entirely devoted to examples. We first discuss the case of a weakly sound class $\Psi$ (Section~\ref{sound-companion}); in this case the colimit types are actually classes of weights and, unlike in the general case, we are able to give a corresponding weakening of orthogonality. The main results of the section are Theorems~\ref{sound-strongcharact} and \ref{sound-weakcharacterization}. Next we consider the class of wide pullbacks (Section~\ref{widepbks}); here we recover the classical results on locally polypresentable categories and compare our results with those given in \cite{HT96qc:articolo}. In Section~\ref{products+G-powers} we consider a generalization of the weakly locally presentable categories to the enriched setting, while in Section~\ref{flexible} we discuss the case of accessible 2-categories with flexible limits.

Finally, in Section~\ref{wrsnc}, we compare reflectivity with respect to a particular colimit type $\mt C$ with the existing notions of weak reflectivity. We do this to obtain the characterization theorems of \cite{AR94:libro} for weakly locally presentable categories, and of \cite{LR12:articolo} for accessible 2-categories with flexible limits, as instances of our theory.

\newpage

\section{The general setting}\label{general-psi}

We fix a locally presentable and symmetric monoidal closed category $\V=(\V_0,\otimes,I)$ as our base of enrichment; in particular $\V$ will be locally $\lambda$-presentable as a closed category in the sense of Kelly \cite{Kel82:articolo}, for some regular cardinal $\lambda$ (by \cite[2.4]{KL2001:articolo}). Any other regular cardinal taken into consideration will be greater or equal than this $\lambda$. The enriched notion of accessibility that we study is the one first introduced in \cite{BQR98} and more recently developed in \cite{LT22:virtual}.

Our convention is that a {\em weight} will always mean a $\V$-functor
$M\colon\C\op\to\V$ with small domain. Notwithstanding this, it still
makes sense to consider weighted colimits $M*S$ for a $\V$-functor
$S\colon\C\to\K$ and presheaf $M\colon\C\op\to\V$ with $\C$ large, and
indeed this colimit will exist if $M$ is small and $\K$ is
cocomplete.

From now on we shall consider a class $\Psi$ of weights representing the kind of limits that our accessible $\V$-categories will be assumed to have.  Such a class $\Psi$ is said to be {\em locally small} \cite{KS05:articolo} if the free completion under $\Psi$-limits of a small $\V$-category is small. We do not assume this, and indeed it is not the case for key examples such as (the classes corresponding to) connected limits, products, or wide pullbacks.

We denote completions under $\Psi$-colimits and under $\Psi$-limits by
$\Psi\A$ and $\Psi^\dagger\A$ respectively. When $\Psi=\P$ is the
class of all weights, we recover the free completions $\P\A$ and
$\P^\dagger\A$ under small colimits and limits. We generally
write $Y\colon\A\to\Psi\A$ and $Z\colon\A\to\Psi^\dagger\A$ for the
respective (Yoneda) embeddings. 

\begin{observation}\label{virtual-col}
	Let $\B$ be a full subcategory of $\P^\dagger \A$ containing
        the representables, and $J\colon \A\to\B$ the inclusion.
        Then $\A$ admits the colimit $F*S$ for a $\V$-functor
        $S\colon\D\to\A$ and weight $F\colon\D\op\to\V$ just when the
        induced $\V$-functor $[\D\op,\V](F,\A(S,-))\colon\A\to\V$ is representable. On
        the other hand, to say that $\B$ admits the colimit $F*JS$ is
        to say that $[\D\op,\V](F,\A(S,-))$ lies in $\B$; in this
        case $[\D\op,\V](F,\A(S,-))$ {\em is} the colimit $F*JS$.
\end{observation}

For any such $\B$, we therefore obtain a relaxed cocompleteness
condition on $\A$ saying that $\B$ admits weighted colimits of
diagrams landing in $\A$. We shall consider various such weakenings
throughout the paper. The first type of these is the topic of this
section, in which we work with those $\V$-functors $\A\to\V$ which are
{\em $\Psi$-precontinuous}, as defined in Definition~\ref{Psi-precont}
below. One can think of $\Psi$-precontinuity as a way
of defining preservation of $\Psi$-limits without assuming the
existence of these limits in $\A$. We use it in Theorem~\ref{relative-quasiflattheorem} below to give a first characterization of accessible $\V$-categories with $\Psi$-limits. 

To begin with, consider a {\em small} $\V$-category $\A$; we are interested in those $\V$-functors $M\colon\A\op\to\V$ for which $\tx{Lan}_YM\colon[\A,\V]\to\V$ preserves $\Psi$-limits of diagrams landing in $\A\op$, where $Y\colon\A\op\to[\A,\V]$ is the Yoneda embedding. When $\A$ is $\Psi$-cocomplete, and so $\A\op$ is $\Psi$-complete, this is just requiring that $M$ be a $\Psi$-continuous $\V$-functor.

Since $\tx{Lan}_YM\cong M*-$, the condition above is saying that $M$-weighted colimits commute in $\V$ with $\Psi$-limits of representable $\V$-functors. That is, for any $N\colon\D\to\V$ in $\Psi$ and $H\colon\D\to\A\op$ the canonical map defines an isomorphism
\begin{align}\label{commutativity}
	M*\{N, YH\}\cong\{N,M*YH\}.
\end{align}
This approach turns out still to be useful when the $\V$-category $\A$
is not assumed to be small, but $M\colon\A\op\to\V$ is a small
$\V$-functor. In fact, in this case the collection of $\V$-functors
$[\A,\V]$ does not form a $\V$-category in general, but colimits of
arbitrary $\V$-functors $\A\to\V$ weighted by $M$ do exist since $M$
is small, and weighted limits of $\V$-functors are always defined pointwise in $\V$. Therefore, both sides of (\ref{commutativity}) still exist, as does the canonical comparison, and we can give the following definition:

\begin{Def}\label{Psi-precont}
	Let $\Psi$ be a class of weights, $\A$ be a $\V$-category, and
        $M\colon\A\op\to\V$ a small $\V$-functor. We say that $M$ is
        {\em $\Psi$-precontinuous} if $M$-weighted colimits commute in
        $\V$ with $\Psi$-limits of representable $\V$-functors; in
        other words, if the canonical map defines an
        isomorphism~\eqref{commutativity} for all $N\colon\D\to\V$ in
        $\Psi$ and all $H\colon\D\to\A\op$. We write
        $\Psi\tx{-PCts}(\A\op,\V)$ for the $\V$-category of all small $\Psi$-precontinuous $\V$-functors. 
\end{Def}

Note that we have inclusions
$\A\subseteq\Psi\tx{-PCts}(\A\op,\V)\subseteq\P\A$ so that
$\Psi\tx{-PCts}(\A\op,\V)$ is indeed a $\V$-category. 

In the rest of this section we consider what happens when we weaken
{\em representability} of various functors to {\em
  $\Psi$-precontinuity}.

\begin{es}
	If $\Psi$ is a locally small and sound class of weights in the sense of \cite[Definition~3.2]{LT22:virtual} then $\Psi$-precontinuous and $\Psi$-flat $\V$-functors on a small $\V$-category $\A$ coincide by definition; hence $\Psi\tx{-PCts}(\A\op,\V)=\Psi^+\A$ for any small $\A$, where $\Psi^+$ is the class of all the $\Psi$-flat weights. The case where $\A$ is any large $\V$-category is treated in Section~\ref{sound-companion}.
\end{es}

\begin{prop}\label{precont=cont}
	If $M\colon\A\op\to\V$ is $\Psi$-precontinuous then it preserves any existing $\Psi$-limits. On the other hand, if $\A$ is $\Psi$-cocomplete, then a small $M$ is $\Psi$-precontinuous if and only if it is $\Psi$-continuous. 
\end{prop}
\begin{proof}
	The Yoneda embedding $Y\colon\A\op\hookrightarrow[\A,\V]$ preserves any existing $\Psi$-limits in $\A\op$ and $M*Y-\cong M$. Therefore if $M$ is $\Psi$-precontinuous then it preserves any $\Psi$-limit that happens to exist. If $\A$ is $\Psi$-cocomplete then $\Psi$-limits of representables are still representables so $\Psi$-continuity implies $\Psi$-precontinuity.
\end{proof}

\begin{cor}
	The inclusion $V\colon\A\to\Psi\tx{-PCts}(\A\op,\V)$ preserves any existing $\Psi$-colimits.
\end{cor}
\begin{proof}
	This says that for any $M\in\Psi\tx{-PCts}(\A\op,\V)$ the $\V$-functor $\Psi\tx{-PCts}(\A\op,\V)(V-,M)$ preserves any existing $\Psi$-limits. But $\Psi\tx{-PCts}(\A\op,\V)(V-,M)\cong M$, so that follows by the proposition above.
\end{proof}

\begin{cor}
	For any $\Psi$-cocomplete $\V$-category $\A$ we have an equality
	$$ \Psi\tx{-PCts}(\A\op,\V)=\P\A\cap \Psi\tx{-Cont}(\A\op,\V) $$%
	so that $\Psi$-precontinuous $\V$-functors out of $\A\op$ coincide with the small $\Psi$-continuous $\V$-functors. Moreover $V\colon\A\to\Psi\tx{-PCts}(\A\op,\V)$ is $\Psi$-cocontinuous.
\end{cor}
\begin{proof}
	By Proposition~\ref{precont=cont}, if $\A$ is $\Psi$-cocomplete then a small $\V$-functor $M\colon\A\op\to\V$ is $\Psi$-precontinuous if and only if it is $\Psi$-continuous.
\end{proof}

As anticipated at the beginning of the section, we shall now consider a relaxed cocompleteness condition for a $\V$-category $\A$: 

\begin{Def} 
	Given a $\V$-category $\A$, a weight $M\colon\C\op\to\V$ with small domain, and $H\colon\C\to\A$, we say that the {\em $\psicheck$-virtual colimit} of $H$ weighted by $M$ exists in $\A$ if $[\C\op,\V](M,\A(H,-))$ is a small $\Psi$-precontinuous $\V$-functor. We say that $\A$ is {\em $\psicheck$-virtually cocomplete} if it has all $\psicheck$-virtual colimits.
\end{Def}

By Observation~\ref{virtual-col} applied to
$\B=\Psi\tx{-PCts}(\A,\V)\op$, a $\V$-category $\A$ is
$\psicheck$-virtually cocomplete if and only if
$\Psi\tx{-PCts}(\A,\V)\op$ has colimits of representables. When
$\Psi=\emptyset$ is the class with no weights, a
$\emptyset$-precontinuous $\V$-functor is simply a small $\V$-functor;
thus we recover the notion of {\em virtual cocompleteness} defined in
\cite[Definition~4.15]{LT22:virtual}: this says that $\P^\dagger\A$
has colimits of representables.

\begin{prop}\label{cocomplete+psicomplete}
  The following are equivalent for a $\Psi$-complete
  $\V$-category $\A$:
  \begin{enumerate}\setlength\itemsep{0.25em}
  \item $\A$ is virtually cocomplete;
  \item $\P^\dagger\A$ is cocomplete;
  \item $\Psi\tx{-PCts}(\A,\V)\op$ is cocomplete;
  \item $\A$ is $\psicheck$-virtually cocomplete.
  \end{enumerate}
\end{prop}
\begin{proof}
 
  For any $\V$-category $\A$, the equivalence $(1)\Leftrightarrow (2)$
  holds by the dual of \cite[Theorem~3.8]{DL07}, while
  $(3)\Rightarrow(4)$ trivially holds since if
  $\Psi\tx{-PCts}(\A,\V)\op$ has colimits it has colimits of
  representables.

  If $\A$ is $\Psi$-complete, then $\Psi\tx{-PCts}(\A,\V)\op$ consists
  of the small and $\Psi$-continuous $\V$-functors, and so is closed
  in  $\P^\dagger\A$  under colimits. Thus $(2)\Rightarrow(3)$ and
  $(4)\Rightarrow(1)$.
\end{proof}

A $\V$-functor $F\colon\A\to\K$ has a left adjoint if each $\K(X,F-)\colon\A\to\V$ is representable; while it has a virtual left adjoint if each $\K(X,F-)$ is small. Similarly:

\begin{Def}
	We say that a $\V$-functor $F\colon\A\to\K$ has a $\psicheck$-virtual left adjoint if, for each $X\in\K$, the $\V$-functor $\K(X,F-)$ is $\Psi$-precontinuous. If $F$ is fully faithful we then say that $\A$ is $\psicheck$-virtually reflective in $\K$.
\end{Def}

In other words, $F$ has a $\psicheck$-virtual left adjoint if and only if, for each $X\in\K$, the $\V$-functor $\K(X,F-)$ has a relative left $V$-adjoint, where $V\colon\A\hookrightarrow\Psi\tx{-PCts}(\A,\V)\op$ is the inclusion. If $\Psi=\emptyset$ is the empty class, $\emptyset^\vee$-virtual left adjoints coincide with virtual left adjoints.

\begin{prop}
	The following are equivalent for a $\Psi$-complete and virtually cocomplete $\V$-category $\A$, and a $\V$-functor $F\colon\A\to\V$:\begin{enumerate}
		\item $F$ is small and $\Psi$-continuous;
		\item $F$ has a $\psicheck$-virtual left adjoint.
	\end{enumerate}
\end{prop}	
\begin{proof}
	A $\psicheck$-virtual left adjoint $L\colon\V\to \Psi\tx{-PCts}(\A,\V)\op$ exists if and only if  the $\V$-functor $[X,F-]\colon\A\to\V$ is small and $\Psi$-precontinuous, and in that case is given by $LX:=[X,F-]$. Now, if $F$ is small and $\Psi$-continuous it is in particular $\Psi$-precontinuous; thus $[X,F-]$ is $\Psi$-precontinuous too (since $[X,-]$ preserves all limits) and small (being the copower of $F$ by $X$ in $\P^\dagger\A$, which is cocomplete). Conversely, if $F$ has a $\psicheck$-virtual left adjoint then $F\cong[I,F-]$ is in $\Psi\tx{-PCts}(\A,\V)\op$, and therefore it is small and $\Psi$-continuous.
\end{proof}

In the accessible case we then obtain:

\begin{cor}\label{psi-cont}
	The following are equivalent for a $\V$-functor $F\colon\A\to\K$ between $\Psi$-complete accessible $\V$-categories:\begin{enumerate}
		\item $F$ is accessible and $\Psi$-continuous;
		\item $F$ has a $\psicheck$-virtual left adjoint.
	\end{enumerate}
\end{cor}
\begin{proof}
	This follows from the proposition above plus the fact that every accessible $\V$-category is virtually cocomplete, and that, for $\A$ an accessible $\V$-category, $\K(X,F-)\colon\A\to\V$ is accessible if and only if it is small (\cite[Proposition~4.9]{LT22:virtual}).
\end{proof}

  \begin{prop}\label{prop:what-is-A}
    Let $\K$ be an accessible $\V$-category and $\A$ an accessible and
    accessibly-embedded subcategory of $\K$. Then there exists a regular
    cardinal $\alpha$ such that the following are equivalent for any $L\in\K$
    \begin{enumerate}
    \item $L$ is in $\A$;
    \item $ZL$ is orthogonal in $\P^\dagger\K$, for all $K\in\K_\alpha$, to the map $ZK \to
      \{\K(K,J-),ZJ\}$ induced by $Z\colon
      \K(K,J-)\to(\P^\dagger\K)(ZK,ZJ-)$; 
    \item the canonical $\K(K,J-)*\K(J-,L)\to\K(K,L)$ is invertible
      for all $K\in\K_\alpha$.
    \end{enumerate}
  \end{prop}

  \begin{proof}
    The equivalence of (1) and (2) holds by (the proof of)
    \cite[Proposition~4.29]{LT22:virtual}. Condition (2) says that the
    induced map
    \[ (\P^\dagger\K)(\{\K(K,J-),ZJ\},ZL)\to (\P^\dagger\K)(ZK,ZL) \]
    is invertible, but by Yoneda this is just (3).
  \end{proof}

\begin{teo}\label{strong-quasiflattheorem}
	Let $\K$ be an accessible $\V$-category with $\Psi$-limits and $\A$ a full subcategory of $\K$. The following are equivalent:\begin{enumerate}\setlength\itemsep{0.25em}
		\item $\A$ is accessible, accessibly embedded, and closed under $\Psi$-limits;
		\item $\A$ is accessibly embedded and $\psicheck$-virtually reflective.
	\end{enumerate}
\end{teo}
\begin{proof}
  $(1)\Rightarrow (2)$. This follows by Corollary~\ref{psi-cont} above.
  
  $(2)\Rightarrow (1)$. Since $\A$ is $\psicheck$-virtually
  reflective it is also virtually reflective, and so accessible
  by \cite[Corollary~4.24]{LT22:virtual}. Thus we need only
  prove that $\A$ is closed in $\K$ under $\Psi$-limits.

Suppose then that $N\colon\D\to\V$ is in $\Psi$ and
$S\colon\D\to\A$. We must prove that the limit $\{N,JS\}$ lies
in $\A$. Now
$\K(K,J-)$ is small and $\Psi$-precontinuous for all $K\in\K$, so 
\begin{align*}
    \K(K,J-)*\K(J-,\{N,JS\}) 
          &\cong \K(K,J-)*\{N,\K(J,JS)\} \tag{defn of limit} \\
          &\cong \K(K,J-)*\{ N, YS\} \tag{$J$ fully faithful} \\
          &\cong \{N,\K(K,JS) \} \tag{by \eqref{commutativity}} \\
          &\cong \K(K,\{N,JS\}) \tag{defn of limit} 
\end{align*}
and so $\{N,JS\}\in\A$ by Proposition~\ref{prop:what-is-A}.
\end{proof}

\begin{teo}\label{relative-quasiflattheorem}
	The following are equivalent for a $\V$-category $\A$: \begin{enumerate}\setlength\itemsep{0.25em}
		\item $\A$ is accessible and $\Psi$-complete;
		\item $\A$ is accessible and $\Psi\tx{-PCts}(\A,\V)\op$ is cocomplete;
		\item $\A$ is accessible and $\psicheck$-virtually cocomplete;
		\item $\A$ is accessibly embedded and $\psicheck$-virtually reflective in $[\C,\V]$ for some $\C$.
	\end{enumerate}
	In that case $\Psi\tx{-PCts}(\A,\V)$ consists of the small $\Psi$-continuous $\V$-functors.
\end{teo}
\begin{proof}
	Since $\A$ is $\psicheck$-virtually cocomplete just
        when $\Psi\tx{-PCts}(\A,\V)\op$ has colimits of
        representables, $(2)\Rightarrow(3)$ is trivial. $(4)\Rightarrow (1)$ is a consequence of Theorem~\ref{strong-quasiflattheorem}.
	
	$(1)\Rightarrow (2)$. $\A$ is accessible and therefore virtually cocomplete; thus it follows from Proposition~\ref{cocomplete+psicomplete} that $\Psi\tx{-PCts}(\A,\V)\op$ is cocomplete.
	
	$(3)\Rightarrow (4)$. Let $\alpha$ be such that $\A$ is
        $\alpha$-accessible; then take $\C=\A_{\alpha}\op$ so that we
        have an accessible embedding $J\colon
        \A\hookrightarrow[\C,\V]$. There is a $\V$-functor $G\colon
        \C\op\to \Psi\tx{-PCts}(\A,\V)\op$ which, up to isomorphism,
        is just the inclusion $\A_{\alpha}\subseteq\A\subseteq
        \Psi\tx{-PCts}(\A,\V)\op$. Since $\Psi\tx{-PCts}(\A,\V)\op$
        has colimits of objects of $\A$ and $G$ lands in $\A$ by
        construction, it has an essentially unique cocontinuous extension $L\colon [\C,\V]\to \Psi\tx{-PCts}(\A,\V)\op$ which is a left $V$-adjoint to $J$, as desired.
\end{proof}

\section{Main results}\label{main-psi}

\subsection{Colimit types and companions}\label{colimit-types}

Theorem~\ref{relative-quasiflattheorem} bears some resemblance to
Theorem~\ref{companioin-intro}. Nonetheless, the notion of
$\Psi$-precontinuous $\V$-functor is not enough to capture the known
characterization theorems for locally polypresentable and weakly
locally presentable categories. To obtain them we need some more
explicitly determined class of presheaves for our relaxed notion of
representability.

Towards this end, we shall introduce notions of {\em colimit type} and
{\em companion} related to a given class $\Psi$ of limits. Given such
a companion $\mt{C}$ for $\Psi$, we shall construct, for each
$\V$-category $\A$, a full subcategory 
$\mt{C}_1\A$ of $\Psi\tx{-PCts}(\A\op,\V)$ containing the representables; in fact $\mt{C}_1\A$  will turn out to be the same as
$\Psi\tx{-PCts}(\A\op,\V)$ when $\A$ is accessible and $\Psi$-cocomplete, but not in general.

These notions of colimit type and companion will allow us to capture things like colimits of free groupoid actions which, as we saw in the introduction, arise in the characterization theorem of locally polypresentable categories.

\begin{Def} A {\em colimit type} $\mt{C}$ is the data of a full replete subcategory 
	$$\mt{C}_{ M }\hookrightarrow[\C,\V]$$%
	for any weight $ M \colon\C\op\to\V$. This may equivalently be given by specifying the class 
	$$\mt C =\{ (M,H)\ |\ H\in\mt C_M \}.$$%
\end{Def}

Let us see some examples.

\begin{es}\label{soundcompanion}
	If $\Phi$ is a class of weights there is a colimit type $\mt C^\Phi$ with
	\[ \mt C^\Phi_M = \begin{cases} 
		[\C,\V] & \tx{if } M\in[\C\op,\V]\tx{ is in }\Phi, \\
		\emptyset & \tx{otherwise.} 
	\end{cases}
	\]
\end{es}

\begin{es}\label{freegroupoids}
	For $\V=\bo{Set}$, consider the colimit type $\mt F$ defined by: $H\in\mt F_M$, for $M\colon\C\op\to\bo{Set}$, if and only if $\C$ is a groupoid, $M=\Delta 1$, and $H\colon\C\to\bo{Set}$ is free, in the sense that 
	\begin{center}
		
		\begin{tikzpicture}[baseline=(current  bounding  box.south), scale=2]
			
			\node (a) at (-0.6,0) {$0$};
			\node (b) at (0.1,0) {$HA$};
			\node (c) at (1,0) {$HB$};
			
			\path[font=\scriptsize]
			
			(a) edge [->] node [above] {} (b)
			([yshift=1.5pt]b.east) edge [->] node [above] {$Hf$} ([yshift=1.5pt]c.west)
			([yshift=-1.5pt]b.east) edge [->] node [below] {$Hg$} ([yshift=-1.5pt]c.west);
		\end{tikzpicture}
	\end{center}
	is an equalizer for any $f,g\colon A\to B$ with $f\neq g$. 
\end{es}

\begin{es}\label{pseudoequivalencerelation}
	For $\V=\bo{Set}$, consider the colimit type $\mt R$ defined by: $H\in\mt R_M$, for $M\colon\C\op\to\bo{Set}$, if and only if $\C=\{x\rightrightarrows y\}$ is the free category on a pair of arrows, $M=\Delta 1$, and $H\colon \C\to\bo{Set}$ is a pseudo equivalence relation, in the sense that the pair of functions identified by $H$ factors as
	\begin{center}
		
		\begin{tikzpicture}[baseline=(current  bounding  box.south), scale=2]
			
			\node (a) at (-0.6,0) {$Hx$};
			\node (b) at (0.1,0) {$Z$};
			\node (c) at (0.9,0) {$Hy$};
			
			\path[font=\scriptsize]
			
			(a) edge [->>] node [above] {$e$} (b)
			([yshift=1.5pt]b.east) edge [->] node [above] {$h$} ([yshift=1.5pt]c.west)
			([yshift=-1.5pt]b.east) edge [->] node [below] {$k$} ([yshift=-1.5pt]c.west);
		\end{tikzpicture}
	\end{center}
	an epimorphism $e$ followed by a kernel pair $(h,k)$. Such a factorization, when it exists, is unique since it will be given by the epi-mono factorization of the induced $Hx\to Hy\times Hy$. An equivalent definition is \cite[Definition~6]{CV98:articolo} which explains why these are called pseudo equivalence relations. 
	
\end{es}

Given a class of weights $\Psi$ and a colimit type $\mt C$, we express
the commutativity of $\Psi$-limits with colimits of diagrams of
type $\mt C$ as follows:

\begin{Def}
	Let $\Psi$ be a class of weights and $\mt{C}$ be a colimit
        type; we say that {\em $\mt{C}$ is compatible with $\Psi$} if,
        for any presheaf $M\colon\C\op\to\V$ with $\mt C_M$
        non-empty, $\mt C_M\subseteq[\C,\V]$ is closed under $\Psi$-limits and the composite
	\begin{center}
		
		\begin{tikzpicture}[baseline=(current  bounding  box.south), scale=2]
			
			\node (a) at (-0.8,0) {$\mt C_M$};
			\node (b) at (0,0) {$[\C,\V]$};
			\node (c) at (1,0) {$\V$};
			
			\path[font=\scriptsize]
			
			(a) edge [right hook->] node [above] {} (b)
			(b) edge [->] node [above] {$M*-$} (c);
		\end{tikzpicture}
	\end{center}
	preserves them.
\end{Def}

\begin{obs}
	For the purposes of this section it would be enough to require
        that $M*-$ preserve $\Psi$-limits of diagrams landing in $\mt
        C_M$; however, since the condition above is satisfied by all
        our examples and will be required anyway for the main results
        of Section~\ref{sketches-C}, we opted for that. See in
        particular Proposition~\ref{phisketchisasketch}. As we have
        seen in the examples above, $\mt C_M$ is often the empty
        $\V$-category; since these are not always $\Psi$-complete (for
        instance when $\Psi$-limits imply the existence of a terminal
        object), we could not require the closure of $\mt C_M$ under $\Psi$-limits for the non-empty $\mt C_M$. 
\end{obs}

\begin{es}$ $
	\begin{enumerate}
		\item If $\mt C=\mt C^\Phi$, for a class of weights $\Phi$, the compatibility condition says that $\Phi$ consists of (some or all) $\Psi$-flat weights (see Section~\ref{sound-companion}).
		\item The colimit type $\mt F$ of free groupoid actions is compatible with wide pullbacks (see Section~\ref{widepbks}).
		\item The colimit type $\mt R$ of pseudo equivalence relations is compatible with small products (see Section~\ref{products+G-powers}).
	\end{enumerate}
\end{es}

With the following definition we introduce the $\V$-categories
$\mt{C}_1 \A$ and $\mt{C}_1^\dagger \A$ which  will be used to
define our $\mt{C}$-virtual notions.

\begin{Def}\label{C_1A}
	Let $\mt C$ be a colimit type and $\A$ be a $\V$-category. We define $\mt{C}_1\A$ to be the full subcategory of $\P\A$ consisting of:\begin{enumerate}
		\item the representables;
		\item $M*YH$ for any $M\colon\C\op\to\V$ and $H\colon\C\to\A$ for which $\A(A,H-)\colon\C\to\V$ lies in $\mt C_M$ for all $A\in\A$.
	\end{enumerate}
	Dually, let $\mt{C}_1^\dagger \A=\mt{C}_1(\A\op)\op$;  this consists of certain $\V$-functors $F\colon\A\to\V$. 
\end{Def}

\begin{prop}
	If $\mt C$ is a colimit type that is compatible with $\Psi$ then 
	$$\mt C_1\A\subseteq\Psi\tx{-PCts}(\A\op,\V).$$%
	In particular, every $F\colon\A\op\to\V$ in $\mt C_1\A$ preserves any existing $\Psi$-limits, and the inclusion $\A\hookrightarrow\mt C_1\A$ preserves any existing $\Psi$-colimits.
\end{prop}
\begin{proof}
	Consider $X\in\mt C_1\A$; if $X$ is representable it is $\Psi$-precontinuous, so suppose that $X\cong M*YH$ for some $M\colon\C\op\to\V$ and $H\colon\C\to\A$ for which $\A(A,H-)\in\mt C_M$ for all $A\in\A$. We need to show that $X$-weighted colimits commute with $\Psi$-limits of representables. For that, consider $N\colon\D\to\V$ in $\Psi$ and $S\colon\D\to\A\op$; then:
	\begin{align} 
		X*\{N,YS\} &\cong M*YH*\{N,YS\} \tag{defn $X$} \\
			&\cong
                   M-*\{N\square,\A(S\square,H-)\}\tag{Yoneda} \\
			&\cong \{N\square,
                   M-*\A(S\square,H-)\} \label{eq:compatible} \\
			&\cong \{N,(M*YH)*YS)\}\tag{Yoneda}\\
			&\cong \{N,X*YS\}\tag{defn $X$} 
	\end{align}
	where \eqref{eq:compatible} holds since $\A(SD,H-)$ lies in $\mt C_M$ by hypothesis (for any $D\in\D$) and $M*-$ preserves $\Psi$-limits of diagrams landing in $\mt C_M$.
\end{proof}

\begin{cor}
	If $\mt C$ is a colimit type compatible with $\Psi$ and $\A$ is $\Psi$-cocomplete, then any $F\in\mt C_1\A$ is $\Psi$-continuous and small.
\end{cor}

The following definition identifies when, for a given class of weights $\Psi$, a colimit type $\mt C$ is rich enough to capture results in the spirit of Theorem~\ref{strong-quasiflattheorem} and Theorem~\ref{relative-quasiflattheorem}.

\begin{Def}\label{companion}
	We say that a colimit type $\mt{C}$ is a {\em companion} for $\Psi$ if: \begin{enumerate}\setlength\itemsep{0.25em}
		\item[(I)] $\mt{C}$ is compatible with $\Psi$;
		\item[(II)] for any $\Psi$-complete and virtually cocomplete $\A$, each small $\Psi$-continuous $\V$-functor $F\colon\A\to\V$ lies in $\mt C_1^\dagger\A$.
	\end{enumerate}
\end{Def}

Assuming (I), condition (II) is equivalent to saying that for any $\Psi$-complete and virtually cocomplete $\A$ we have $\mt C_1^\dagger\A=\Psi\tx{-PCts}(\A,\V)\op$; that is, $\mt C_1^\dagger\A$ consists of the small $\Psi$-continuous $\V$-functors $F\colon\A\to\V$.

\begin{ese}\label{totalex} $ $\\
	\vspace{-10pt}
	\begin{enumerate}\setlength\itemsep{0.25em}
		\item All weakly sound classes in the sense of Definition~\ref{sound-def}, where the classes of diagrams are actually classes of weights (see Example~\ref{soundcompanion}), give examples of companions --- see Section~\ref{sound-companion}, which also looks at many particular examples.
		
		\item For $\V=\bo{Set}$, the colimit type
                  $\mt F$ given by the free groupoid actions
                  (Example~\ref{freegroupoids}) is a companion for
                  wide pullback diagrams --- see Section~\ref{widepbks}.
		
		\item For $\V=\bo{Set}$, the colimit type $\mt R$ given by the pseudo equivalence relations (Example~\ref{pseudoequivalencerelation}) is a companion for small products --- see Example~\ref{products} and Section~\ref{wrsnc}.
		
		\item More generally, we may consider enriched colimit types similar to the class $\mt R$ above when the class of weights is given by small products and powers by a dense generator --- see Sections~\ref{products+G-powers} and \ref{wrsnc}.
		
		\item For $\V=\bo{Cat}$, the colimit type
                  $\mt P$ of Definition~\ref{2-companion}  is a
                  companion for the class of flexible limits --- see
                  Sections~\ref{flexible} and \ref{wrsnc}. 
		
		\item For $\V=\bo{Set}$, the colimit type $\mt{S}^\lambda$ (including $\lambda=\infty$) of $\lambda$-sifted diagrams is a companion for the class of $\lambda$-small powers --- see Section~\ref{powers}.
	\end{enumerate}
\end{ese}

Proposition~\ref{cocomplete+psicomplete}, Corollary~\ref{psi-cont},
Theorem~\ref{strong-quasiflattheorem}, and
Theorem~\ref{relative-quasiflattheorem} each proves the equivalence of
various conditions involving, either explicitly or implicitly, the
$\V$-category \linebreak $\Psi\tx{-PCts}(\A,\V)\op$.~We now prove a
series of analogues for these results involving $\mt{C}_1^\dagger\A$
in place of $\Psi\tx{-PCts}(\A,\V)\op$. In each case,
$\Psi\tx{-PCts}(\A,\V)\op$ and $\mt{C}_1^\dagger\A$ will turn out to
be equal since $\A$ will be $\Psi$-complete and virtually
cocomplete --- see Remark~\ref{rmk:Cdag-PCts} below --- but we will not know this beforehand. Thus despite the similarity, these new results do not flow automatically from the previous ones.

The first step is to introduce the notion of {\em $\mt C$-virtually cocomplete} $\V$-category $\A$. This generalizes the dual of \cite[Definition~4.3]{karazeris2009representability} from a doctrine to a colimit type $\mt C$.

\begin{Def}\label{C-virtual-col}
	Let $\mt C$ be a colimit type. Given a $\V$-category $\A$, a weight $M\colon\C\op\to\V$ with small domain, and $H\colon\C\to\A$, we say that the {\em $\mt C$-virtual colimit} of $H$ weighted by $M$ exists in $\A$ if $[\C\op,\V](M,\A(H,-))\colon\A\to\V$ lies in $\mt C_1^\dagger\A$. We say that $\A$ is {\em $\mt C$-virtually cocomplete} if it has all $\mt C$-virtual colimits.
\end{Def}

When $\mt C=\emptyset$ is the empty colimit type, a $\V$-category is $\emptyset$-virtually cocomplete if and only if it is cocomplete. When $\V=\bo{Set}$ and $\mt C=\tx{Fam}$ is the class of weights for small products, a category is $\tx{Fam}$-virtually cocomplete if and only if it is multicocomplete \cite{Die80:articolo}. Similarly, we shall see that if $\mt C=\mt F$ is the colimit type of free groupoid actions, then a category is $\mt F$-virtually cocomplete if and only if it is polycocomplete (dual of Proposition~\ref{polycomplete}).

\begin{prop}\label{cocomplete}
	Let $\mt C$ be a companion for $\Psi$. The following are equivalent for a $\Psi$-complete $\V$-category $\A$:\begin{enumerate}\setlength\itemsep{0.25em}
		\item $\A$ is virtually cocomplete;
		\item $\P^\dagger\A$ is cocomplete;
		\item $\mt{C}_1^\dagger\A$ is cocomplete;
		\item $\A$ is $\mt C$-virtually cocomplete.
	\end{enumerate}
\end{prop}
\begin{proof}
	Note that, since $\mt{C}_1^\dagger \A$ contains the representables, all existing colimits in $\mt{C}_1^\dagger \A$ are computed pointwise, so that the inclusion $\mt{C}_1^\dagger  \A\hookrightarrow\P^\dagger \A$ always preserves any existing colimits. Thus the implications $(3)\Rightarrow (4)\Rightarrow (1)$ are trivial, while $(1)\Rightarrow(2)\Rightarrow(3)$ follow from Proposition~\ref{cocomplete+psicomplete} since in this case $\mt C_1^\dagger\A=\Psi\tx{-PCts}(\A,\V)\op$.
\end{proof}

Once again we can relax the notion of left adjoint by refining
the notion of virtual left adjoint to the case of a colimit type: 

\begin{Def}\label{C-virtual-left}
	We say that a $\V$-functor $F\colon\A\to\K$ has a {\em $\mt
          C$-virtual left adjoint} if the $\V$-functor $\K(X,F-)$
        lies in $\mt C_1^\dagger\A$ for each $X\in\K$. If $F$ is fully faithful we then say that $\A$ is {\em $\mt{C}$-virtually reflective} in $\K$.
\end{Def}

In other words, $F$ has a $\mt C$-virtual left adjoint if and only if it has a relative left $V$-adjoint, where $V\colon\A\hookrightarrow\mt{C}_1^\dagger \A$ is the inclusion.

\begin{obs}\label{rmk:Cdag-PCts}
	Since the inclusion $\mt C_1^\dagger\A\subseteq\Psi\tx{-PCts}(\A,\V)\op$ always holds, any $\V$-functor $F\colon\A\to\K$ having a $\mt C$-virtual left adjoint also has a $\psicheck$-virtual left adjoint, and any $\mt C$-virtually reflective subcategory of a $\V$-category $\K$ is $\psicheck$-virtually reflective. The converse holds whenever $\A$ is $\Psi$-complete and virtually cocomplete since then $\mt C_1^\dagger\A=\Psi\tx{-PCts}(\A,\V)\op$.
\end{obs}

If in place of the solution-set condition we assume accessibility, the following result captures several relative versions of the adjoint functor theorem \cite{freyd1964abelian,kainen1971weak,diers1977categories,Lam89:PhD, solian1990pluri,BLV:articolo}. See Sections~\ref{Examples-companions} and~\ref{wrsnc} where we compare our left $\mt C$-adjoints with those appearing in the literature.

\begin{prop}
	Let $\mt C$ be a companion for $\Psi$. The following are equivalent for a $\V$-functor $F\colon\A\to\K$ between $\Psi$-complete accessible $\V$-categories:\begin{enumerate}
		\item $F$ is accessible and $\Psi$-continuous;
		\item $F$ has a left $\mt C$-adjoint.
	\end{enumerate}
\end{prop}
\begin{proof}
	This is a consequence of Corollary~\ref{psi-cont} plus the fact that $\mt C_1^\dagger\A=\Psi\tx{-PCts}(\A,\V)\op$ under these assumptions.
\end{proof}

\begin{teo}\label{C-strong-charact}
	Let $\mt C$ be a companion for $\Psi$, and $\K$ be an accessible $\V$-category with $\Psi$-limits. The following are equivalent for a full subcategory $\A$ of $\K$:\begin{enumerate}\setlength\itemsep{0.25em}
		\item $\A$ is accessible, accessibly embedded, and closed under $\Psi$-limits;
		\item $\A$ is accessibly embedded and $\mt{C}$-virtually reflective.
	\end{enumerate}
\end{teo}
\begin{proof}
	$(1)\Rightarrow(2)$ by the proposition above, while $(2)\Rightarrow(1)$ is a consequence of Theorem~\ref{strong-quasiflattheorem} plus the fact that $\mt C^\dagger\A \subseteq \Psi\tx{-PCts}(\A,\V)\op$.
\end{proof}

\begin{teo}\label{relatice-C-charact}
	Let $\mt C$ be a companion for the class $\Psi$. The following are equivalent for a $\V$-category $\A$:\begin{enumerate}\setlength\itemsep{0.25em}
		\item $\A$ is accessible and $\Psi$-complete;
		\item $\A$ is accessible and $\mt{C}_1^\dagger \A$ is cocomplete;
		\item $\A$ is accessible and $\mt C$-virtually cocomplete;
		\item $\A$ is accessibly embedded and $\mt{C}$-virtually reflective in $[\C,\V]$ for some $\C$.
	\end{enumerate}
\end{teo}
\begin{proof}
	Use Theorem~\ref{C-strong-charact} and argue as in the proof of Theorem~\ref{relative-quasiflattheorem}.
\end{proof}

\vspace{5pt}

\subsection{Sketches}\label{sketches-C}

In this section we treat a notion of model of a sketch which differs
from the usual one; to justify and better understand this notion it
might be helpful to see models of sketches as {\em morphisms} in the
category of sketches. This can be described as the category whose
objects are sketches $\S=(\B,\mathbb{L},\mathbb{C})$ and morphisms
$F\colon\S\to\S'$ are functors $F\colon\B\to\B'$ which send the
classes of cylinders $\mathbb{L}$ and cocylinders $\mathbb{C}$ to the classes $\mathbb{L'}$ and $\mathbb{C'}$ respectively. 

Denote by $\V_\P$ the {\em large} sketch based on $\V$ itself and with the two specified classes given by all the limiting and colimiting cylinders in $\V$ (here we are allowing the base of our sketch to be large). Then, under this notation, a model of a sketch $\S=(\B,\mathbb{L},\mathbb{C})$ is just a morphism of sketches $F\colon\S\to\V_\P$. 

Consider now a colimit type $\mt C$ and recall that, given a weight $M\colon \D\op\to\V$, we denote by $\mt{C}_{M}\subseteq[\D,\V]$ the full subcategory of those $H$ for which $(M,H)\in\mt{C}$. To introduce the notion of $\mt C$-model consider instead of $\V_\P$ the sketch $\V_{\mt C}$ given by $\V$ together with the class $\mathbb{L}_\P$ of all limiting cylinders and the class  $\mathbb{C}_\mt C$ of all colimiting cylinders $\eta\colon  M \Rightarrow\V(H-,X)$ for which $H\in\mt C_M$. Then we define a $\mt C$-model of a sketch $\S=(\B,\mathbb{L},\mathbb{C})$ to be a morphism $F\colon\S\to\V_{\mt C}$ in the category of sketches. More explicitly:

\begin{Def}\label{C-models}
	Let $\mt{C}$ be a colimit type and $\S=(\B,\mathbb{L},\mathbb{C})$ a sketch.
	A {\em $\mt C$-model} of $\S$ is a $\V$-functor $F\colon \B\to\V$ satisfying the following conditions:\begin{enumerate}\setlength\itemsep{0.25em}
		\item[(i)] for every $\gamma$ in $\mathbb{L}$, its image $F\gamma$ is a limiting cylinder in $\V$;
		\item[(ii)] for every $\eta$ in $\mathbb{C}$, its image $F\eta$ is a colimiting cocylinder in $\V$;
		\item[(iii)] for every $\eta\colon  M \Rightarrow\B(H-,B)$ in $\mathbb{C}$, the functor $FH$ lies in $\mt{C}_{ M }$.
	\end{enumerate}
	Denote by $\tx{Mod}_\mt{C}(\S)$ the full subcategory of $[\B,\V]$ spanned by the $\mt C$-models of $\S$ in $\V$.
\end{Def}

In other words a $\mt C$-model is a model of $\S$ (in the usual sense) which in addition satisfies condition (iii).

\begin{obs}\label{Phi-colim-sketch}
	When $\mt C=\mt C^\Phi$ is the colimit type defined by a class of weights $\Phi$, then $\V$-categories of $\mt C^\Phi$-models of (general) sketches and $\V$-categories of models of limit/$\Phi$-colimit sketches are the same. In fact, given a sketch $\S=(\B,\mathbb{L},\mathbb{C})$, if $\mathbb{C}$ contains a weight which is not in $\Phi$, then $\tx{Mod}_{\mt C^\Phi}(\S)=\emptyset$; while, if all weights appearing in $\mathbb{C}$ lie in $\Phi$, then $\tx{Mod}_{\mt C^\Phi}(\S)=\tx{Mod}(\S)$. 
\end{obs}

Accessible $\V$-categories with $\Psi$-limits can be seen as $\V$-categories of $\mt C$-models:

\begin{prop}\label{strong-psi-sketch}
	Let $\mt C$ be a companion for $\Psi$, and let $\A$ be
        accessible, accessibly embedded, and closed under
        $\Psi$-limits in $[\C,\V]$. Then there exist a fully faithful
        $H\colon \C\hookrightarrow\B$ and a sketch $\S$ on $\B$ for
        which $\tx{Mod}_\mt{C}(\S)=\tx{Mod}(\S)$ and  $\tx{Ran}_H$ induces an equivalence $$\A\simeq\tx{Mod}_\mt{C}(\S).$$
\end{prop}
\begin{proof}
    Let $\K=[\C,\V]$ and write $J\colon\A\to\K$ for the inclusion.
    By Proposition~\ref{prop:what-is-A} we know that
there is a regular cardinal $\alpha$ such that $L\in\K$ lies in $\A$
if and only if the canonical $\K(K,J-)*\K(J-,L)\to\K(K,L)$ is
invertible, for all $K\in\K_\alpha$. By
Theorem~\ref{strong-quasiflattheorem} we know that $\K(K,J-)$ is
$\Psi$-precontinuous, and so has the form $\tx{Lan}_{H_K}M_K$ for some $M_K\colon\D\op_K\to\V$ and some
$H_K\colon\D\op_K\to\A$ with $\A(H_K-,A)\in\mt{C}_{M_K}$ for all
$A\in\A$. Choose $\beta\ge\alpha$ such that $\K_\beta$ contains the
image of $JH_K\colon\D\op_K\to\K$ for all $K\in\K_\alpha$.

Let $\B=\K\op_\beta$ with $J'\colon\B\to\K\op$ the inclusion, and let $\mathbb{L}$ consist of all the
$\beta$-small limiting cylinders of $\B$. Then $\tx{Ran}_H$ induces an
equivalence between $\K=[\C,\V]$ and $\tx{Mod}(\B,\mathbb{L})$. Since the
composite $JH_K\colon\D\op_K\to\K$ lands in $\K_\beta$, it induces a $\V$-functor
$H'_K\colon \D_K\to\B$, and now the isomorphism $\tx{Lan}_{H_K}M_K\cong
\K(K,J-)$ corresponds to a composite
\[ M_K \rightarrow \K(K,JH_K-) \cong \B(H'_K-,K) \]
which we call $\eta_K$. We define $\mathbb{C}$ to consist of these
$\eta_K\colon M_K\to \B(H_K-,K)$.

The models of $(\B,\mathbb{L})$ have the form $\K(J'-,L)$ for some
$L\in\K$. Such a $\K(J'-,L)$  sends $\eta_K$ to a colimit in $\V$ just
when the composite
\[ M_K \rightarrow \K(K,JH_K-) \rightarrow [\K(JH_K-,L),\K(K,L)] \]
corresponds to an isomorphism
\[ M_K * \K(JH_K,L) \cong \K(K,L) \]
and so to an isomorphism
\[ \tx{Lan}_{H_K}M_K * \K(J-,L)\cong \K(K,L) \]
which, given that $\tx{Lan}_{H_K}M_K\cong \K(K,J-)$ says that $L\in\A$.

Thus $\tx{Ran}_H$ induces an equivalence $\A\simeq\tx{Mod}(\S)$. The
fact that every model is a $\mt{C}$-model follows from the fact that
each $\A(H_K-,A)\in\mt{C}_{M_K}$. 
\end{proof}

Note that, for arbitrary $\mt C$ and $\S$, it is not necessarily true that $\tx{Mod}_\mt{C}(\S)$ is accessible since it may not be a $\V$-category of models in the usual sense.

\begin{obs}
	For each $M$, the $\V$-category $\mt{C}_M$ itself can be expressed as $\tx{Mod}_\mt{C}(\S)$ for a sketch $\S$: given $M\colon \D\op\to\V$, consider it as the colimit $M\cong M*Y$ in $[\D\op,\V]$ and denote by $\B$ the full subcategory of $[\D\op,\V]$ spanned by the representables and $M$; let $W\colon\D\hookrightarrow\B$ be the inclusion. If we consider the colimit cocylinder $\eta\colon  M \to\B(W-,M)$, then
	$$ \mt{C}_M\simeq \tx{Mod}_\mt{C}(\S)$$%
	where $\S=(\B,\mathbb{L}=\emptyset,\mathbb{C}=\{\eta\})$ and the equivalence is obtained left Kan extending along the inclusion $W$.
\end{obs}

Thus, if we want any $\tx{Mod}_\mt{C}(\S)$ to be accessible for any
sketch $\S$, we should at least ask $\mt{C}_M$ to be accessible and
accessibly embedded in its ambient $\V$-category, for each weight
$M$. And it turns out that this suffices: 

\begin{Def}
	Let $\mt C$ be a companion for $\Psi$; we say that $\mt C$ is
        an {\em accessible companion} for $\Psi$ if for each weight $M\colon \D\op\to\V$ the $\V$-category $\mt{C}_{ M }$ is accessible and accessibly embedded in $[\D,\V]$.
\end{Def}

Recall that $\mt C_M$, when non-empty, is assumed to be closed under $\Psi$-limits in $[\C,\V]$.

\begin{es}\label{acccomp} The following are examples of accessible companions:
  \begin{enumerate}\setlength\itemsep{0.25em}
  \item for every weakly sound class $\Psi$, the companion $\mt C^{\Psi^+}$ given by the $\Psi$-flat weights: indeed $\mt C^{\Psi^+}_M$ is either empty or the whole presheaf $\V$-category.
  \item $\V=\bo{Set}$ and the companion $\mt F$, of free
    groupoid diagrams, for the class of wide pullbacks:
    for each groupoid $\G$ consider the category $\G'$
    obtained from $\G$ by adding an initial
    object $0$. On $\G'$ consider the sketch with limit
    diagrams in $\mathbb{L}$ the pairs $g,h\colon G\to
    H$, for every morphism  $g\neq h$ in $\G$, and with cone specification the unique arrow $0\to G$. The only colimit diagram in $\mathbb{C}$ is the empty one with empty cocone given by the object $0$. Let $J\colon \G\to \G'$ be the inclusion; then restriction along $J$ induces an equivalence 
    \[ \tx{Mod}(\G', \mathbb{L},\mathbb{C})\simeq\mt F_\G. \]
		On one hand, if $F\colon \G'\to\bo{Set}$
                is a model then $FJ$ is clearly a free groupoid
                action; conversely, if $F\colon \G\to\bo{Set}$ is in
                $\mt F_\G$ then define $F'\colon \G'\to\bo{Set}$
                by  extending $F$ with $F'(0)=\emptyset$. This $F'$
                is a model of the sketch since $F$ is a free
                groupoid action. It follows that each $\mt F_\G$ is accessible and closed under filtered colimits in $[\G,\bo{Set}]$. Since the only colimit specification in the sketch is an initial object, $\mt F_\G$ is closed under (all connected limits, and in particular) wide pullbacks.
		
		\item $\V=\bo{Set}$ and the companion $\mt R$, of pseudo-equivalence relations, for the class of products: let $\C$ be the category with two parallel non-identity arrows $f,g\colon X\to Y$, then $\mt R_M$ is non empty only for $M=\Delta 1\colon\C\op\to\bo{Set}$. Consider the category $\C'$ in $\P\C$ spanned by: $\C$, the coequalizer $q$ of $(f,g)$, the kernel  pair $h,k\colon Z\to Y$ of $q$, and the kernel pair $(h',k')$ of the map $e\colon X\to Z$ induced by the kernel pair $(h,k)$. Then define the sketch on $\C'$ with limit conditions $\mathbb{L}$ saying that $(h,k)$ and $(h',k')$ are the kernel pairs of $q$ and $e$ respectively; the only colimit conditions $\mathbb{C}$ are saying that $q$ and $e$ are the coequalizers of $(h,k)$ and $(h',k')$ respectively. It's then easy to see that 
		$$\mt R_{\Delta 1}\simeq \tx{Mod}(\C',\mathbb{L},\mathbb{C}),$$ 
		and as a consequence it is accessible and closed under filtered colimits and products in $[\C,\bo{Set}]$. 
		
		\item See also examples from Section~\ref{products+G-powers} and \ref{flexible}.
	\end{enumerate}
\end{es}

\begin{prop}\label{phisketchisasketch}
	Let $\mt C$ be an accessible companion for $\Psi$. For any sketch $\S=(\B,\mathbb{L},\mathbb{C})$ the $\V$-category $\tx{Mod}_\mt{C}(\S)$ of $\mt C$-models of $\S$ is accessible, accessibly embedded, and closed under $\Psi$-limits in $[\B,\V]$.
\end{prop}
\begin{proof}
	First note that we can see $\tx{Mod}_\mt{C}(\S)$ as the intersection in $[\B,\V]$ of all the $\tx{Mod}_\mt{C}(\S_\eta)$, for all $\eta\in\mathbb{C}$ and $\S_\eta=(\B,\mathbb{L},\{\eta\})$. Now, each $\tx{Mod}_\mt{C}(\S_\eta)$ can be seen as the pullback
	\begin{center}
		\begin{tikzpicture}[baseline=(current  bounding  box.south), scale=2]
			
			\node (a0) at (0,0.8) {$\tx{Mod}_\mt{C}(\S_\eta)$};
			\node (b0) at (1.3,0.8) {$\mt C_M$};
			\node (c0) at (0,0) {$\tx{Mod}(\B,\mathbb{L})$};
			\node (d0) at (1.3,0) {$[\C,\V]$};
			\node (e0) at (0.16,0.6) {$\lrcorner$};
			
			\path[font=\scriptsize]
			
			(a0) edge [->] node [above] {} (b0)
			(a0) edge [->] node [left] {} (c0)
			(b0) edge [right hook->] node [right] {} (d0)
			(c0) edge [->] node [below] {$-\circ H$} (d0);
		\end{tikzpicture}	
	\end{center}
	with $\eta\colon  M \Rightarrow\B(H-,B)$, $M\colon\C\op\to\V$, and $H\colon\C\to\B$. Since the inclusion of $\mt C_M$ in $[\C,\V]$ is an isofibration, the $\V$-categories involved are accessible and $\Psi$-complete, and the $\V$-functors are accessible and $\Psi$-continuous, it follows by \cite[Corollary~5.7]{LT22:virtual} that each $\tx{Mod}_\mt{C}(\S_\eta)$ is accessible, accessibly embedded, and closed under $\Psi$-limits in $[\B,\V]$. For the same reason $\tx{Mod}_\mt{C}(\S)$ is also an accessible and $\Psi$-complete $\V$-category.
\end{proof}

Hence we can characterize accessible $\V$-categories with $\Psi$-limits as $\V$-categories of $\mt C$-models:

\begin{teo}\label{Psi-sketch}
	Let $\Psi$ be a class of weights, $\mt C$ be an accessible companion for $\Psi$, and $\A$ be a $\V$-category; the following are equivalent:\begin{enumerate}\setlength\itemsep{0.25em}
		\item $\A$ is accessible with $\Psi$-limits;
		\item $\A$ is equivalent to the $\V$-category of $\mt C$-models of a sketch.
	\end{enumerate}
\end{teo}
\begin{proof}
	Put together Proposition~\ref{strong-psi-sketch} and \ref{phisketchisasketch}.
\end{proof}

\subsubsection{$\mt C$-sketches:}\label{D-sketch}

Instead of using the notion of $\mt C$-model of a general sketch, we can also introduce a notion of $\mt C$-sketch whose $\V$-categories of (standard) models characterize accessible $\V$-categories with $\Psi$-limits. In general this notion is more technical than that of $\mt C$-model (as we see below); however in the case of pseudo-equivalence relations and free groupoid actions we recover the notions of limit/epi and galoisian sketches.

Let $\mt C$ be an accessible companion for $\Psi$. By accessibility of $\mt C$, for any $M\colon\D\op\to\V$ we can fix a fully faithful $W_M\colon\D\hookrightarrow\D_M$ and a sketch $\S_M=(\D_M,\mathbb{L}_M,\mathbb{C}_M)$ on $\D_M$ together with an equivalence
\begin{equation}\label{sketch-for-C}
	-\circ W_M\colon\tx{Mod}(\S_M)\longrightarrow \mt C_M.
\end{equation}
Using this we define the notion of $\mt C$-sketch as follows:

\begin{Def}\label{C-sketch-model}
	A {\em $\mt C$-sketch} $\T$ is determined by a sketch $\S=(\B,\mathbb{L},\mathbb{C})$ for which every cocylinder in $\mathbb{C}$ is of the form 
	$$\eta\colon N\Rightarrow\B(H-,B)\colon\D_M\op\to\V $$%
	and such that $N\cong\tx{Lan}_{W_M\op}M$, for some weight $M$. 
	Then $\T$, as a sketch, is given by $\S$ together with the additional classes of cylinders $H(\mathbb{L}_M)$ and of cocylinders $H(\mathbb{C}_M)$, for each $\eta\in\mathbb{C}$ and $M$ as above. A model of $\T$ is then a model of
	$$ (\B, \mathbb{L}\sqcup_{\eta\in\mathbb{C}} H(\mathbb{L}_M), \mathbb{C}\sqcup_{\eta\in\mathbb{C}} H(\mathbb{C}_M)) $$%
	in the standard sense. 
\end{Def}

The following lemma will be important for the characterization theorem.

\begin{lema}\label{extendmodels}
	The following are equivalent for an accessible $\V$-category $\A$, a weight $M\colon\D\op\to\V$, and a diagram $H\colon\D\op\to \A$: \begin{enumerate}\setlength\itemsep{0.25em}
		\item $\A(H-,A)\in\mt C_M$ for any $A\in\A$;
		\item $H\op$ can be extended to a model $\widehat{H}\op\colon\D_M\to\P(\A\op)$ of $\S_M$ in $\P(\A\op)$.
	\end{enumerate} 
\end{lema}
\begin{proof}
	It's clear that if $\widehat H\op$ extends $H\op$ and is a model of $\S_M$ then $H$ lies representably in $\mt C_M$. Conversely, if $H\colon\D\op\to \A$ lies representably in $\mt C_M$, then the transpose $S\colon\A\to[\D,\V]$ of $YH\op\colon\D\to\P(\A\op)$ is accessible and lands in $\mt C_M$. Then we can compose it with the given equivalence to obtain an accessible $\V$-functor $T\colon\A\to[\D_M,\V]$ that lands in $\tx{Mod}(\S_M)$. Transposing again we obtain a $\V$-functor $\widehat H\op\colon \D_M\to\P(\A\op)$ which extends $H\op$ and is a model of $\S_M$.
\end{proof}

We are now ready to prove that accessible $\V$-categories with $\Psi$-limits can equivalently be described as $\V$-categories of models of a $\mt C$-sketch:

\begin{teo}\label{C-sketch}
	Let $\Psi$ be a class of weights, and $\mt C$ be an accessible companion with a fixed sketch presentation as in (\ref{sketch-for-C}). A $\V$-category $\A$ is accessible with $\Psi$-limits if and only if it is equivalent to the $\V$-category of models of a $\mt C$-sketch.
\end{teo}
\begin{proof}
	Let $\T$ be a $\mt C$-sketch as in Definition~\ref{C-sketch-model}; then a $\V$-functor $F\colon\B\to\V$ is a model of $\T$ if and only if it is a $\mt C$-model of the sketch $\S=(\B,\mathbb{L},\mathbb{C}')$, where $\mathbb{C}'=\{\eta\circ W_M\op\}_{\eta\in\mathbb{C}}$ (with $\eta$ as in Definition~\ref{C-sketch-model}). Thus $\tx{Mod}(\T)=\tx{Mod}_\mt C(\S)$ is accessible with $\Psi$-limits by Theorem~\ref{Psi-sketch}. 
	
        Conversely, given an accessible $\V$-category $\A$ with $\Psi$-limits, we consider a fully faithful and accessible $J\colon\A\to\K=[\C,\V]$ and work in the setting of the proof of Proposition~\ref{strong-psi-sketch}. The diagrams $H_K\colon\D_K\op\to\A$ considered there satisfy condition $(1)$ of Lemma~\ref{extendmodels}; thus they can be extended to some $\widehat{H}_K\colon\D_{M_K}\op\to\P^\dagger\A$ such that $\widehat{H}_K\op$ is a model of $\S_{M_K}$. Then, by possibly replacing $\C$ with a larger (but still small) $\V$-category, we can assume that every $\P^\dagger J\circ \widehat{H}_K\colon \D_{M_K}\op\to\P^\dagger\K$ lands in $\K$.
	
	In addition, we take the $\V$-category $\B\op\subseteq \K$ (considered in the proof) to contain the images of all the $\widehat{H}_K$; call the resulting $\V$-functors $\widehat{H}_K'\colon\D_{M_K}\to\B$. Now consider the sketch $\S=(\B,\mathbb{L},\mathbb{C})$ given in Proposition~\ref{strong-psi-sketch}, and define the $\mt C$-sketch $\T$ with same base $\V$-category $\B$, same set of cylinders $\mathbb{L}$, and cocylinders
	$$ \widehat\gamma_K\colon \tx{Lan}_{W_{M_K}\op}M_K\to\B(\widehat H'_K-,K)$$%
	induced by the $\gamma_X$ specified in \ref{strong-psi-sketch} for each given $K$. It's then easy to see that $\tx{Mod}_{\mt C}(\S)= \tx{Mod}(\T)$, and hence $\A$, being equivalent to $\tx{Mod}_{\mt C}(\S)$, is the $\V$-category of models of a $\mt C$-sketch. 
\end{proof}

\begin{es}\label{galoisian}
	In the case of the companion $\mt F$ for wide pullbacks ($\V=\bo{Set}$), the remark above says that the sketches characterizing accessible categories with wide pullbacks have the following restrictions on the colimit cocones:\begin{enumerate}
		\item there is always an empty cocone, whose vertex will be denoted by $0$;
		\item any other cocone $H\to \Delta C$ is indexed on a groupoid $\G$ and comes together with a cone specification
		\begin{center}
			
			\begin{tikzpicture}[baseline=(current  bounding  box.south), scale=2]
				
				\node (a) at (0.5,0.7) {$0$};
				\node (b) at (0,0) {$HA$};
				\node (c) at (1,0) {$HB$};
				
				\path[font=\scriptsize]
				
				(a) edge [dashed, ->] node [above] {} (b)
				(a) edge [dashed, ->] node [above] {} (c)
				([yshift=1.5pt]b.east) edge [->] node [above] {$Hf$} ([yshift=1.5pt]c.west)
				([yshift=-1.5pt]b.east) edge [->] node [below] {$Hg$} ([yshift=-1.5pt]c.west);
			\end{tikzpicture}
		\end{center}
		for any parallel pair $f\neq g$ in $\G$. 
	\end{enumerate}
	By seeing every groupoid as (equivalent to) a coproduct of
        groups, these sketches can easily be recognized as the {\em galoisian} sketches of Ageron \cite[Definition~4.14]{Age92:articolo}.
\end{es}

\begin{es}\label{limit-epi}
	We can apply the same arguments to the companion $\mt R$ for products ($\V=\bo{Set}$). By the remark above, the only restriction for colimit cocones appearing in the sketches for accessible categories with products is that they need to be coequalizers of kernel pairs. Since we are working in $\bo{Set}$, this is equivalent to the specification of a set of maps which need to be sent to epimorphisms. That's exactly the data of a limit/epi sketch where all the cocones are as below.
	\begin{center}
		\begin{tikzpicture}[baseline=(current  bounding  box.south), scale=2]
			
			\node (a) at (0,0) {$B$};
			\node (b) at (0.8,0) {$C$};
			\node (a') at (0,-0.8) {$C$};
			\node (b') at (0.8,-0.8) {$C$};

			\path[font=\scriptsize]
			
			(a) edge[->] node [above] {$e$} (b)
			(a') edge[dashed, ->] node [below] {$1_C$} (b')
			(a) edge[->] node [left] {$e$} (a')
			(b) edge[dashed, ->] node [right] {$1_C$} (b');
			
		\end{tikzpicture}
	\end{center}
	Thus we recover the characterization of \cite[Theorem~4.13]{AR01}. The same argument can be applied in the more general setting of Section~\ref{products+G-powers}.
\end{es}

\subsection{The $\bo{Set}$-case}\label{set-comp}

For this section we take $\V$ to be $\bo{Set}$ and give a further characterization of companions.
To do so we need to make the following assumption:

\begin{as}\label{acccon}
	All the components $\mt C_M$ of the colimit type $\mt C$ are assumed to be accessible, accessibly embedded, and such that the sketches defining them involve only colimits and connected limits.
\end{as}

All the colimit types in Example~\ref{acccomp} satisfy this condition.

The following lemma is a straightforward extension of the well-known special case where $\A$ is small.

\begin{lema}\label{small-slice}
	Let $F\colon\A\to\bo{Set}$ be a small functor and $\tx{El}(F)$ be its category of elements, so that we have a projection $q\colon\tx{El}(F)\to\A$. Then 
	$$ \P(\tx{El}(F)\op)\simeq\P(\A\op)/F $$%
	and, under this equivalence,
        $\tx{Lan}_q\colon\P(\tx{El}(F)\op)\to\P(\A\op)$ corresponds to the projection $Q\colon\P(\A\op)/F\to\P(\A\op)$.
\end{lema}
\begin{proof}
  	For the proof it's easier to see $\tx{El}(F)\op$
        as the slice category $\A\op/F$, which comes together a fully
        faithful inclusion $J\colon\A\op/F\to\P(\A\op)/F $.
Since $\P(\A\op)$ is cocomplete, so too is $\P(A\op)/F$, while the
projection $Q$ is cocontinuous. Since any object of $\P(\A\op)$ is
canonically a colimit of representables, so too any object of
$\P(\A\op)/F$ is canonically a colimit of a diagram in the image of
$J$. Likewise since homming out of representables is cocontinuous, so
too homming out of objects in the image of $\V$ is cocontinuous. This
implies that $\P(\A\op)/F$ is the free cocompletion of $\A\op/F$ and
so is equivalent to $\P(\A\op/F)$.
\end{proof}

The following then provides an easy way to recognize companions:

\begin{prop}\label{companion-terminal}
	Let $\mt C$ be a colimit type as in Assumption~\ref{acccon} and $\Psi$ a class of weights. Then $\mt C$ is a companion for $\Psi$ if and only if:
	\begin{enumerate}\setlength\itemsep{0.25em}
		\item[(i)] $\mt{C}$ is compatible with $\Psi$;
		\item[(ii)] for any $\Psi$-complete and virtually cocomplete $\B$, the category $\mt C_1^\dagger\B$ has an initial object.
	\end{enumerate}
\end{prop}
\begin{proof}
  Condition (i) is identical to condition (I) in the definition of
  companion, so we just need to compare (ii) with (II).
  
  First observe that an initial object in either $\mt{C}^\dagger_1\B$ or $\P^\dagger\B$ can only be $\Delta 1\colon\B\to\bo{Set}$ and will exist just when this $\Delta 1$ lies
  in the category in question.
    
  If $\mt C$ is a companion for $\Psi$, so that (II) holds, observe that $\Delta 1\colon\B\to\bo{Set}$
  preserves any existing limits, so is $\Psi$-continuous; while since $\B$
  is virtually cocomplete,  $\mt{C}^\dagger_1\B$ has an initial
  object, which must be $\Delta 1$, which is therefore small. Thus
  $\Delta 1$ is small and $\Psi$-continuous so, by 
  condition (II) in the definition of companion, lies in  $\mt{C}^\dagger_1\B$. 

Suppose conversely that (i) and (ii) hold, that 
$\A$ is $\Psi$-complete and virtually cocomplete, and that
$F\colon\A\to\bo{Set}$ is small and $\Psi$-continuous. We must show
that $F\in\mt{C}^\dagger_1\A$. 

Since $\A$ is $\Psi$-complete and $F$ is $\Psi$-continuous, $\tx{El}(F)$
is $\Psi$-complete and the projection $q\colon\tx{El}(F)\to\A$ is
$\Psi$-continuous. On the other hand, since $\A$ is virtually cocomplete, $\P(\A\op)$ is
cocomplete, hence so too is $\P(\A\op)/F\simeq\P(\tx{El}(F)\op)$, and thus
$\tx{El}(F)$ is virtually cocomplete. It follows by (ii) that 
$\mt{C}^\dagger_1(\tx{El}(F))$ has an initial object, necessarily
equal to $\Delta 1\colon \tx{El}(F)\to\bo{Set}$. Thus we can write
$\Delta 1$ as $M*YS$ for some $M\colon\D\to\bo{Set}$ and $S\colon
\D\op\to\tx{El}(F)\op$, where $\tx{El}(F)(S-,E)\in\mt{C}_M$ for all
$E\in\tx{El}(F)$; equivalently, where the composite
$YS\colon\D\op\to[\tx{El}(F),\bo{Set}]$ is a model of the sketch $\S_M$
corresponding to $\mt{C}_M$.

Now $F=M*YqS$, and so it will suffice to
show that $\A(qS-,A)\in\mt{C}_M$ for all $A\in\A$, or equivalently
that $YqS\colon\D\op\to[\A,\bo{Set}]$ is a model of the sketch
$\S_M$. But $YqS=\tx{Lan}_q.YS$, so it will suffice to show that
$\tx{Lan}_q\colon[\tx{El}(F),\bo{Set}]\to[\A,\bo{Set}]$ preserves models
of $\S_M$.  This $\tx{Lan}_q$ is of course cocontinuous, but it also
preserves connected limits since up to equivalence it is given by the
projection $[\A,\bo{Set}]/F\to [\A,\bo{Set}]$. We have assumed that
the sketch $\S_M$ can be chosen so that the only limits it uses are
connected ones, thus $\tx{Lan}_q$ does indeed preserve models of $\S_M$.

\end{proof}

\section{Examples}\label{Examples-companions}

\subsection{The weakly sound case}\label{sound-companion}

In this section we focus on the case where our companion $\mt C$, for a class of weights $\Psi$, is determined by a class of weights $\Phi$; that is $\mt C$ coincides with the colimit type $\mt C^\Phi$ defined by $\mt C^\Phi_M=[\C,\V]$ if $M\colon\C\op\to\V$ is in $\Phi$, and $\mt C^\Phi_M=\emptyset$ otherwise (Example~\ref{soundcompanion}). For simplicity we denote the colimit type simply by $\Phi$ instead of $\mt C^\Phi$, but note that whenever we say that $\Phi$ is a companion we mean that the corresponding $\mt C^\Phi$ is one.

We sometimes write $\Phi^\dagger_1\A$ or $(\mt C^\Phi)^\dagger_1\A$
for the full subcategory of $\P^\dagger\A$ consisting of the
representables and the $\Phi$-limits of representables; in general
this is smaller than the free completion $\Phi^\dagger\A$ of $\A$
under $\Phi$-limits, but they agree if $\Phi$ is {\em presaturated},
in the sense that $\Phi^\dagger_1\A$ is closed in $\P^\dagger\A$ under
$\Phi$-limits, as will be the case for all examples considered in this
section, and which for simplicity we henceforth suppose.

\begin{obs}
	The only place in this section where presaturation is actually needed is Proposition~\ref{phi-virtual-adj}, which is not used elsewhere. In fact, if $\Phi$ is not presaturated then Theorems~\ref{sound-strongcharact} and \ref{sound-weakcharacterization} below still hold; one just has to replace $\Phi^\dagger\A$ with $\Phi^\dagger_1\A$ in all the definitions and statements.\\
	The reason why we did assume our $\Phi$ to be presaturated is that, in the context of classes of weights, it seems more natural to work with free completions under all $\Phi$-limits, rather than with 1-step completions (see for instance Proposition~\ref{soundcompchar}). 
\end{obs}

In this context Definition~\ref{companion} translates into the following: let $\Psi$ and $\Phi$ be classes of weights; then $\Phi$ is a companion for $\Psi$ if and only if:\begin{itemize}
	\item[(I)] $\Phi$-colimits commute in $\V$ with $\Psi$-limits;
	\item[(II)] for any $\Psi$-complete and virtually cocomplete
          $\A$, each small $\Psi$-continuous $\V$-functor
          $F\colon\A\to\V$ is a $\Phi$-colimit of representables.
\end{itemize}

Recall that, unlike in the case of the $\Psi$-accessible $\V$-categories of \cite{LT22:virtual}, we don't assume our class $\Psi$ to be locally small; in other words, completions under $\Psi$-limits of small $\V$-categories might be large.

\begin{obs}\label{small-psi}
	If $\Psi$ is a {\em locally small} class of weights and $\Phi$ is a class of weights compatible with $\Psi$, then it's enough to check the companion property (II) above only for small $\V$-categories. Indeed, suppose that the condition holds for any small $\V$-category, let $\A$ be any (possibly large) $\Psi$-complete $\V$-category, and suppose that $F\colon\A\to\V$ is small and $\Psi$-continuous. Then $F=\tx{Lan}_J(FJ)$, with $J\colon\C\hookrightarrow\A$ small; since $\Psi$ is locally small we can suppose that $\C$ is closed in $\A$ under $\Psi$-limits, and hence that $FJ$ is $\Psi$-continuous. By hypothesis then $FJ$ is a $\Phi$-colimit of representables, but $\tx{Lan}_J(-)$ preserves colimits as well as the representables, so $F$ is a $\Phi$-colimit of representables too.
\end{obs}

We begin by extending the definition of flat $\V$-functors and soundness to a possibly large class of weights $\Psi$.

\begin{Def}
	Let $\Psi$ be a class of weights. We say that a small presheaf $M\colon\A\op\to\V$ is {\em $\Psi$-flat} if $M$-colimits commute in $\V$ with $\Psi$-limits. We call {$\Psi$-flat weight} a $\Psi$-flat presheaf with small domain. Denote by $\Psi^+$ the class given by the $\Psi$-flat weights.
\end{Def}

As noted at the beginning of Section~\ref{general-psi}, when
$M\colon\A\op\to\V$ is a small $\V$-functor, $M$-weighted colimits
exist in any cocomplete $\V$-category, making the notion above well defined.

\begin{Def}\label{sound-def}
	A class of weights $\Psi$ is called {\em weakly sound} if every $\Psi$-continuous and small $\V$-functor $ M \colon\A\to\V$, from a virtually cocomplete and $\Psi$-complete $\A$, is $\Psi$-flat.
\end{Def}

When $\Psi$ is locally small, thanks to the same argument as in Remark~\ref{small-psi}, we recover the usual notion of weakly sound class of \cite{LT22:virtual} that only involves small $\Psi$-complete $\V$-categories.

The relationship between $\Psi$ being weakly sound and it having a companion identified by a class of weights is explained by the following proposition. 

\begin{prop}\label{wsound comp}
	If a class $\Phi$ of weights is a companion for $\Psi$ then $\Psi$ is a weakly sound class and $\Phi\subseteq\Psi^+$. Conversely, if $\Psi$ is a weakly sound class then $\Psi^+$ is a companion for $\Psi$.
\end{prop}
\begin{proof}
	For the first assertion consider $M\colon \A\to\V$ to be small and $\Psi$-continuous (with $\A$ virtually cocomplete and $\Psi$-complete); then $M$ is a $\Phi$-colimit of representables since $\Phi$ is a companion for $\Psi$. As a consequence $M*-$ is a $\Phi$-colimit of evaluation functors, which are continuous. Then, since $\Phi$-colimits commute with $\Psi$-limits in $\V$, it follows at once that $M*-$ preserves $\Psi$-limits, and hence that $\Psi$ is weakly sound.
	
	For the second part we already know, by definition, that $\Psi^+$-colimits commute in $\V$ with $\Psi$-limits; therefore consider again $M\colon \A\to\V$ small and $\Psi$-continuous, we need to prove that it is a $\Psi^+$-colimit of representables. By assumption, $M\cong \tx{Lan}_J(MJ)$ for some small full subcategory $J\colon \C\to\A$. Therefore, since $M$ is $\Psi$-flat (being $\Psi$-continuous), $MJ$ is $\Psi$-flat as well by \cite[Lemma 3.3]{LT22:virtual} and hence $M\cong MJ*ZJ$ (where $Z$ is the inclusion of $\A\op$ in $\P(\A\op)$) is a $\Psi^+$-colimit of representables. This shows that $\Psi^+$ is a companion for $\Psi$. 
\end{proof}

As a consequence, given a companion $\Phi$ for $\Psi$ and any $\V$-category $\A$ the following inclusions hold:
$$\Phi\A \subseteq\Psi^+\A ;$$%
and that becomes an equality whenever $\A\op$ is $\Psi$-complete and virtually cocomplete. If in addition every $\Psi$-flat $\V$-functor is a $\Phi$-colimit of representables, then the equality 
$$ \Phi\A = \Psi^+\A $$%
always holds.

\begin{obs}
	It's not true in general, given a companion $\Phi$ for $\Psi$, that all $\Psi$-flat $\V$-functors are $\Phi$-colimits of representables. For example, when $\V=\bo{Ab}$, the filtered colimits form a companion for the class of finite limits, but flat $\bo{Ab}$-functors are not just filtered colimits of representables. In fact they are filtered colimits of finite sums of representables (see \cite[Theorem~3.2]{ObRo70:articolo} and \cite{LT21:articolo}).
\end{obs}

\begin{obs}\label{Rmk:why-not-Psi+} 
  Thus if $\Psi$ is weakly sound, we could just work with
  $\Psi^+$. But using a smaller companion can give better results. For
  a particularly simple example, if $\Psi=\P$ is the class of all
  small limits, then rather than using $\Psi^+=\Q$, we can use
  $\Phi=\emptyset$. Then we recover directly the fact that an
  accessible category is complete if and only if it is cocomplete,
  rather than saying that it is complete if and only if its Cauchy
  completion is cocomplete. See Examples~\ref{soundcomp-examples} for
  further examples. 
\end{obs}

\begin{es}\label{soundcomp-examples} In the following list we give
  examples of weakly sound classes $\Psi$ and spell out when the
  companion can be taken to consist of some, but not all, the
  $\Psi$-flat weights. See \cite[Example~3.6]{LT22:virtual} for
  explanations of why most of these classes are sound. In cases where
  $\V=\bo{Set}$ and we speak of a class of categories rather than a class of weights, we are referring to the corresponding unweighted limits (in other words, weighted by the terminal weight). 
	\begin{enumerate}\setlength\itemsep{0.25em}
		\item $\Psi=\P$ is the class of all small weights. Then $\Psi^+$ is the class of the Cauchy (or absolute) weights. The class $\Phi=\emptyset$ is also a companion for $\P$. 
		\item $\Psi=\emptyset$. Then $\Psi^+=\P$ is the class of all weights. When $\V=\bo{Set}$, the class $\Phi$ of all small categories is also a companion for $\Psi$.  
		\item $\V$ locally $\alpha$-presentable as a closed category, $\Psi$ is the class of $\alpha$-small weights. Then $\Psi^+$ consists of the $\alpha$-flat $\V$-functors. The class $\Phi$ of the $\alpha$-filtered categories is also a companion for $\Psi$; in general not every $\alpha$-flat $\V$-functor is a filtered colimit of representables (see \cite{LT21:articolo}).
		\item $\V$ symmetric monoidal closed finitary quasivariety \cite{LT20:articolo}, $\Psi$ is the class for finite products and finitely presentable projective powers. The class $\Phi$ of sifted categories is a companion for $\Psi$; we do not know whether $\Psi$-flat $\V$-functors are always sifted colimits of representables.
		\item $\V$ cartesian closed, $\Psi$ is the class of finite discrete diagrams. Then $\Psi^+$ consists of those $M$ for which $\tx{Lan}_\Delta M \cong M \times M $ (\cite[Lemma~2.3]{KL93FfKe:articolo}).
		\item $\V=\bo{Set}$, $\Psi$ is the class of connected categories. The class $\Phi$ of discrete categories is a companion for $\Psi$; every element of $\Psi^+$ is a split subobject of coproducts of representables. 
		\item $\V=\bo{Set}$, $\Psi=\{\emptyset\}$. Then $\Psi^+$ is generated by the class of connected categories.
		\item $\V=\bo{Set}$, $\Psi$ consists of the finite connected categories. Then $\Psi^+$ is generated by coproducts of filtered categories.
		\item $\V=\bo{Set}$, $\Psi$ is the class of finite non empty categories. Then $\Psi^+$ is generated by the filtered categories plus the empty category.
		\item $\V=\bo{Set}$, $\Psi$ is the class of finite discrete non empty categories. Then $\Psi^+$ is generated by the sifted categories plus the empty category.
		\item $\V=\bo{Cat}$, $\Psi$ is the class generated by connected conical limits and powers by connected categories. The class $\Phi$ of discrete categories is a companion for $\Psi$; every $\Psi$-flat weight is a split subobject of coproducts of representables.
		\item $\V=([\C\op,\bo{Set}],\otimes,I)$ with the representables closed under the monoidal structure, $\Psi$ is the class defined by powers by representables. The class $\Phi$ of all conical weights is a companion for $\Psi$.
	\end{enumerate}
\end{es}

In the weakly sound context, companions of locally small classes of weights can be characterized as follows. When $\Phi=\Psi^+$ the equivalence of (1), (2), and (3) was first proved in \cite[Theorem~8.11]{KS05:articolo}.

\begin{prop}\label{soundcompchar}
	Let $\Phi$ and $\Psi$ be a pair of classes of weights with $\Psi$ locally small. The following are equivalent:
	\begin{enumerate}\setlength\itemsep{0.25em}
		\item $\Phi$ is a companion for $\Psi$.
		\item $\Phi$-colimits commute in $\V$ with $\Psi$-limits, and for any $\V$-category $\A$ every object of $\P\A$ is a $\Phi$-colimit of objects from $\Psi\A$.
		\item $\tx{Lan}_ZW\colon \Phi(\Psi(-))\to\P(-)$ is an equivalence of endofunctors on $\V\tx{-}\bo{CAT}$, where $W\colon \Psi(-)\to\P(-)$ and $Z\colon \Psi(-)\to\Phi(\Psi(-))$ are the inclusions. 
		\item For each $\Psi$-cocomplete $\A$ the inclusion $Z\colon \A\to\Phi\A$ is $\Psi$-cocontinuous, and freely adding $\Phi$-colimits induces a pseudofunctor $$\Phi(-)\colon \Psi\tx{-}\bo{Coct}\tx{-}\bo{CAT}\longrightarrow\bo{Coct}\tx{-}\bo{CAT}$$%
		from the 2-category of $\Psi$-cocomplete $\V$-categories to that of cocomplete $\V$-categories.
	\end{enumerate}
\end{prop}
\begin{proof}
  $(1)\Rightarrow (2)$. The first part is already in the definition of
  companion. For the second, assume first that $\A$ is small, so that
  $\Psi\A$ is small as well, and consider $F\colon \A\op\to\V$ (an
  object of $\P\A$). Let $J\colon \A\to\Psi\A=\Psi^\dagger (\A\op)\op$
  be the inclusion; then $\tx{Ran}_{J\op}F$ is $\Psi$-continuous (and
  small, since $\Psi\A$ is). The fact that $\Phi$ is a companion for
  $\Psi$ implies that $\tx{Ran}_{J\op}F$ is a $\Phi$-colimit of
  representables: there exist $ M $ in $\Phi\C$ and
  $H\colon \C\to\Psi\A$ such that $\tx{Ran}_{J\op}F\cong M *YH$, where
  $Y\colon\Psi\A\to\P(\Psi\A)$ is the Yoneda embedding. Since
  pre-composition with $J$ defines a cocontinuous $\V$-functor
  $\P(\Psi\A)\to\P\A$ whose composite with
  $Y\colon\Psi\A\to\P(\Psi\A)$ is the inclusion
  $W\colon\Psi\A\to\P\A$, it follows that
  $F\cong (M *YH)J\cong M *WH$, as desired.
	
  Consider now the case of a general $\A$; let $F\colon \A\op\to\V$ be
  small, so that $F\cong \tx{Lan}_H(FH)$ for some small
  $H\colon \B\op\to\A\op$. By the argument above we know that
  $FH\cong M *W_{\B}K$ is a $\Phi$-colimit of objects from $\Psi\B$;
  thus
	$$ F\cong \tx{Lan}_{H}( M *W_{\B}K)\cong  M *\tx{Lan}_{H}W_{\B}K$$%
	but $\tx{Lan}_HW_\B$ takes values in $\Psi\A$, thus $F$ is a
        $\Phi$-colimit of objects from $\Psi\A$.
	
	$(2)\Rightarrow (3)$. Let $\tx{Lan}_ZW\colon
        \Phi(\Psi\A)\to\P\A$ be as in $(3)$. The fact that every
        object of $\P\A$ is a $\Phi$-colimit of objects from $\Psi\A$
        is equivalent to the requirement that $\tx{Lan}_ZW$ be
        essentially surjective. Similarly, we will now show that
        commutativity of $\Phi$-colimits with $\Psi$-limits in $\V$ is
        equivalent to $\tx{Lan}_ZW$ being fully faithful; this will be
        enough to show (3).
	
	Since $\Phi$-colimits commute in $\V$ with $\Psi$-limits, the
        $\Psi$-continuous $M\colon\Psi(\A)\op\to\V$ are closed under
        $\Phi$-colimits. Of course the representables are also
        $\Psi$-continuous; thus any $N\colon\Psi(\A)\op\to\V$ in
        $\Phi(\Psi\A)$ is $\Psi$-continuous, and so the canonical
        $N\to \tx{Ran}_{J\op}(NJ\op)$ is invertible, where
        $J\colon\A\to\Psi\A$ is the inclusion. Note moreover that
        $\tx{Lan}_ZW$ acts by pre-composition with $J\op$; indeed, by
        the arguments above, this defines a $\Phi$-cocontinuous
        $\V$-functor $\Phi(\Psi\A)\to\P\A$ which coincides with $W$
        when restricted to $\Psi\A$. Now for any
        $M,N\colon\Psi(\A)\op\to\V$ in $\Phi(\Psi\A)$ we have
	\begin{equation*}
          \begin{split}
            \Phi(\Psi\A)(M,N)&\cong  [\Psi(\A)\op,\V](M,N)\\
            &\cong  [\Psi(\A)\op,\V](M,\tx{Ran}_{J\op}(NJ\op))\\
            &\cong  [\A\op,\V](MJ\op,NJ\op)\\
          \end{split}
	\end{equation*}
	giving the fully faithfulness of $\tx{Lan}_ZW$.
	
	$(3)\Rightarrow (4)$. Let $\A$ be $\Psi$-cocomplete and
        consider the induced square
	\begin{center}
          \begin{tikzpicture}[baseline=(current bounding box.south),
            scale=2]
			
            \node (c) at (0,0) {$\Phi\A$}; \node (d) at (1.2,0)
            {$\Phi\Psi\A$}; \node (e) at (0.6,0.1) {$\perp$}; \node
            (f) at (0,-0.8) {$\A$}; \node (g) at (1.2,-0.8)
            {$\Psi\A$}; \node (h) at (0.6,-0.7) {$\perp$};
			
            \path[font=\scriptsize]
			
            (c) edge [bend left, <-] node [above] {$\Phi L$} (d) (c)
            edge [right hook->] node [below] {$\Phi J$}(d) (f) edge
            [right hook->] node [left] {$Z$} (c) (g) edge [right
            hook->] node [right] {$Z'$} (d) (f) edge [right hook->]
            node [below] {$J$}(g) (f) edge [bend left, <-] node
            [above] {$L$} (g);
			
          \end{tikzpicture}
	\end{center}
	where $L$ exists since $\A$ is $\Psi$-cocomplete. Note that
        $Z'$ is $\Psi$-cocontinuous because it coincides, up to
        equivalence, with the inclusion
        $\Psi\A\hookrightarrow\P\A$. Thus it is easy to see that $Z$
        must be $\Psi$-cocontinuous as well.
	
	For the second statement, consider the pseudofunctors
        $$U,\Psi(-)\colon \Psi\tx{-}\bo{Coct}\tx{-}\bo{CAT}\to\V\tx{-}\bo{CAT}$$%
	given respectively by the forgetful functor and by freely
        adding $\Psi$-colimits. Since we are restricted to
        $\Psi$-cocomplete categories, the inclusion
        $V\colon U\hookrightarrow \Psi(-)$ has a left adjoint
        $L\colon\Psi(-)\to U$; by applying $\Phi(-)$ this induces an
        adjunction
	\begin{center}
          \begin{tikzpicture}[baseline=(current bounding box.south),
            scale=2]
			
            \node (a) at (0.3,0.8) {$\Phi(-)$}; \node (a') at
            (0.5,0.8) {}; \node (b) at (1.8,0.8) {$\Phi(\Psi(-))$};
            \node (b') at (1.4,0.8) {}; \node (e) at (0.95,0.87)
            {$\perp$};
			
            \path[font=\scriptsize]
			
            ([yshift=3pt]a'.east) edge [bend left, <-] node [above] {}
            ([yshift=3pt]b'.west) ([yshift=-1pt]a.east) edge [right
            hook->] node [below] {} ([yshift=-1pt]b.west);
			
          \end{tikzpicture}
	\end{center}
	but $\Phi(\Psi(-))\simeq\P(-)$. Thus $\Phi\A$ is cocomplete
        whenever $\A$ is $\Psi$-cocomplete, and $\Phi F$ is
        cocontinuous whenever $F$ is $\Psi$-cocontinuous.
	
	$(4)\Rightarrow (1)$. Let us prove first that $\Phi$-colimits
        commute in $\V$ with $\Psi$-limits. Let $ M \colon \C\op\to\V$
        be a weight in $\Phi$, we need to prove that
        $ M *-\colon [\C,\V]\to\V$ is $\Psi$-continuous. For this,
        consider $\A=[\C,\V]\op$; by the assumptions
        $Z\colon \A\to\Phi\A$ is $\Psi$-cocontinuous and hence for
        each $X\in\Phi\A$ the functor $\Phi\A(Z-,X)\colon \A\op\to\V$
        is $\Psi$-continuous. Take now $X:= M *ZY$, where
        $Y\colon \C\to [\C,\V]\op=\A$ is the Yoneda embedding, then
        $$ \Phi\A(Z-,X)\cong  M *\A(-,Y)\cong  M *[\C,\V](Y,-)\cong  M *- $$
	is $\Psi$-continuous, as required.

        It remains to prove property (II) from the definition of
        companion; by Remark \ref{small-psi} we can reduce it to the
        case when $\A$ is small and $\Psi$-complete. Let $\A$ be small
        and $\Psi$-complete, and $F\colon\A\to\V$ be
        $\Psi$-continuous. We need to show that $F$ lies in
        $\Phi^\dagger\A$. To do so, we shall show that the right Kan
        extension $G:=\tx{Ran}_ZF\colon\Phi^\dagger\A\to\V$ is
        representable; then by Yoneda, the representing object must be
        just $F$, which is therefore in $\Phi^\dagger\A$.

        By $(4)$ the $\V$-category $\Phi^\dagger \A=\Phi(\A\op)\op$ is
        complete and $\Phi^\dagger F$ is continuous. Now $G$ is the
        composite of $\Phi^\dagger F$ with the right adjoint to the
        inclusion $Z\colon\B\to\Phi^\dagger\B$, thus $G$ is also
        continuous. 
        
        By \cite[Theorem~4.80]{Kel82:libro}, $G$ will be representable
        provided that $\{G,1_{\Phi^\dagger \A}\}$ exists in $\Phi^\dagger \A$ and is
        preserved by $G$. Since $\A$ is small,  the limit $\{GZ,Z\}$
        exists in $\Phi^\dagger \A$ and
        \[ \{GZ,Z\}\cong\{G,\tx{Ran}_ZZ\}\cong\{G,1_{\Phi^\dagger
            \A}\}.  \]
	Likewise
	\[ G\{G,1_{\Phi^\dagger \A}\}\cong G\{GZ,Z\}\cong \{GZ,GZ\}\cong \{G,\tx{Ran}_Z(GZ)\} \cong \{G,G\}. \]

        Let $G=\tx{Ran}_ZF\colon \Phi^\dagger\A\to\V$. This can be
        constructed as the composite of $\Phi^\dagger F$ and the right
        adjoint to $Z\colon\V\to\Phi^\dagger\V$, and so is also
        continuous.
      \end{proof}

\begin{obs}
	As a consequence of the proof, $\Phi$-colimits commute in $\V$ with $\Psi$-limits if and only if $\tx{Lan}_ZW\colon \Phi(\Psi(-))\to\P(-)$ is pointwise fully faithful.
\end{obs}

\begin{obs}
	Consider the 2-category $\V\tx{-}\bo{CAT}$ of $\V$-categories, $\V$-functors, and $\V$-natural transformations. Then a consequence of point (3) above, is that $\Phi$ is a companion for $\Psi$ if and only if the context $\Phi(-)\to\P(-)$ together with the KZ doctrine $\Psi(-)$ form a GU envelope in the sense of \cite[Definition~3.1]{di2023accessibility}; see also Example~4.4 of the same paper.
\end{obs}

\begin{cor}
	Let $\Psi$ be locally small and $\Phi$ be a companion for $\Psi$. For each $\Psi$-cocomplete $\A$ the $\V$-category $\Phi\A$ is the free cocompletion of $\A$ relative to $\Psi$-colimits.
	In other words the following is a bi-adjunction
	\begin{center}
		\begin{tikzpicture}[baseline=(current  bounding  box.south), scale=2]
			
			\node (a) at (0,0.8) {$\bo{Coct}\tx{-}\bo{CAT}$};
			\node (a') at (0.5,0.8) {};
			\node (b) at (2,0.8) {$\Psi\tx{-}\bo{Coct}\tx{-}\bo{CAT}$};
			\node (b') at (1.4,0.8) {};
			\node (e) at (0.95,0.9) {$\perp$};
			
			\path[font=\scriptsize]
			
			([yshift=3pt]a'.east) edge [bend left, <-] node [above] {$\Phi(-)$} ([yshift=3pt]b'.west)
			([yshift=-1pt]a.east) edge [->] node [below] {$U$} ([yshift=-1pt]b.west);
			
		\end{tikzpicture}	
	\end{center} 
	where $U$ is the forgetful functor. 
\end{cor}

It is now time to introduce the ingredients of the characterization theorem for accessible $\V$-categories with $\Psi$-limits. We begin by introducing $\Phi$-virtual orthogonality classes which generalize the usual notion of orthogonality and that of virtual orthogonality class of \cite{LT22:virtual}.

\begin{Def}\label{Phi-orth}
	Let $\Phi$ be a companion for $\Psi$, and $\K$ be a $\V$-category with inclusion $Z\colon \K\hookrightarrow\Phi^\dagger\K$. Let $f\colon ZX\to P$ be a morphism in $\Phi^\dagger\K$ with representable domain. We say that an object $A$ of $\K$ is {\em orthogonal with respect to $f$} if
	$$ \Phi^\dagger\K(f,ZA)\colon \Phi^\dagger\K(P,ZA)\longrightarrow\Phi^\dagger\K(ZX,ZA) $$%
	is an isomorphism in $\V$. 
\end{Def}

Given an object $P$ in $\Phi^\dagger\K$, we can write it as a $\Phi$-limit of representables $P\cong \{ M ,ZH\}$. Thus to give $f\colon ZX\to P$ is the same as giving a cylinder $\bar{f}\colon  M \to \K(X,H-)$; moreover $\Phi^\dagger\K(ZX,ZA)\cong \K(X,A)$ and $\Phi^\dagger\K(P,ZA)\cong  M *\K(H-,A)$. As a consequence, an object $A$ of $\K$ is orthogonal with respect to $f\colon ZX\to P$ if and only if the map
$$  M *\K(H-,A)\to \K(X,A)$$%
induced by $\bar{f}\colon  M \to \K(X,H-)$ is an isomorphism.

\begin{Def}\label{Phi-virtualorth}
	Let $\Phi$ be a companion for $\Psi$. Given a $\V$-category
        $\K$ and a small collection $\M$ of morphisms in
        $\Phi^\dagger\K$ with representable domain, we
        denote by $\M^\perp$ the full subcategory of $\K$ spanned by
        the objects which are orthogonal with respect to each
        morphism in $\M$. We call {\em $\Phi$-virtual orthogonality class} any full subcategory of $\K$ which arises in this way.
\end{Def}

Equivalently, a $\Phi$-virtual orthogonality class in $\K$ is a
virtual orthogonality class of $\K$ for which the morphisms that
define it (can be chosen to) lie in $\Phi^\dagger\K\subseteq\P^\dagger \K$.

\begin{ese} Let $\K$ be a $\V$-category;
	\begin{itemize}
		\item if $\Psi=\P$ and $\Phi=\emptyset$, then a $\Phi$-virtual orthogonality class in $\K$ is an orthogonality class in the usual sense.
		\item if $\Psi=\emptyset$ and $\Phi=\P$, then a $\Phi$-virtual orthogonality class in $\K$ is a virtual orthogonality class in the sense of \cite[Section~4.5]{LT22:virtual}.
		\item if $\V=\bo{Set}$, $\Psi$ is the class of connected categories, and $\Phi$ that of the discrete ones, then a $\Phi$-virtual orthogonality class in $\K$ is a multiorthogonality class in the sense of \cite{Die80:articolo}.
	\end{itemize}
\end{ese}

Next we specialize the notions of $\mt C$-virtual completeness and $\mt C$-virtual left adjoint of Section~\ref{colimit-types} to the setting of this section; that is, for $\mt C=\mt C^\Phi$. Definition~\ref{C-virtual-col} becomes:

\begin{Def}
	Let $\Phi$ be a companion for $\Psi$. Given a $\V$-category $\A$, a weight $M\colon\C\op\to\V$ with small domain, and $H\colon\C\to\A$, we say that the {\em $\Phi$-virtual colimit} of $H$ weighted by $M$ exists in $\A$ if $[\C\op,\V](M,\A(H,-))\colon\A\to\V$ is a $\Phi$-colimit of representables. We say that $\A$ is {\em $\Phi$-virtually cocomplete} if it has all $\Phi$-virtual colimits.
\end{Def}

While similarly Definition~\ref{C-virtual-left} specializes to:

\begin{Def}
	Let $\Phi$ be a companion for $\Psi$. We say that a $\V$-functor $F\colon \A\to\K$ has a {\em $\Phi$-virtual left adjoint} if for each $X\in\K$ the $\V$-functor $\K(X,F-)$ is a $\Phi$-colimit of representables. If $F$ is fully faithful we say that $\A$ is {\em $\Phi$-virtually reflective} in $\K$.
\end{Def}

\begin{obs}
	This notion was considered before in the ordinary context. For instance, in \cite[Section~2]{tholen1984pro} Tholen says that a functor is a {\em right $\D$-pro adjoint} if, following our notation, it has a $\D$-virtual left adjoint; here $\D$ is a class of indexing categories.
\end{obs}

To say that $\K(X,F-)$ is a $\Phi$-colimit of representables is to say
that it lies in $\Phi^\dagger\A$. The next result relies on our
assumption that $\Phi$ is presaturated.

\begin{prop}\label{phi-virtual-adj}
  For a $\V$-functor $F\colon\A\to\K$, the following are equivalent:
  \begin{enumerate}
  \item $F$ has a $\Phi$-virtual left adjoint;
  \item $F$ has relative left adjoint with respect to the inclusion
    $Z\colon\A\to\Phi^\dagger\A$;
  \item $\Phi^\dagger F\colon\Phi^\dagger\A\to\Phi^\dagger\K$ has a
    left adjoint. 
  \end{enumerate}
\end{prop}

\begin{proof}
  To say that  $L\colon\K\to\Phi^\dagger\A$ is a relative left adjoint
  is to say that we have natural isomorphisms
  \[ (\Phi^\dagger\A)(LX,ZA) \cong \K(X,FA) \]
  so that by Yoneda $LX$ is necessarily given by $\K(X,F-)$, which is
  therefore in $\Phi^\dagger\A$; conversely, if $\K(X,F-)$ is in
  $\Phi^\dagger\A$ it provides the value at $X$ of a relative left
  adjoint. This proves $(1)\Leftrightarrow(2)$.

  By Yoneda, if $\Phi^\dagger\A\to\Phi^\dagger\K$ has a left adjoint,
  it must be given by restriction along $F$; conversely, if
  restriction along $F$ does send objects of $\Phi^\dagger\K$ to
  objects of $\Phi^\dagger\A$ then it defines a left adjoint.

  When this is the case, in particular the restriction $\K(X,F-)$ of
  $ZX$ along $F$ lies in $\Phi^\dagger\A$, giving $(3)\Rightarrow(1)$.

  The converse $(1)\Rightarrow(3)$ is where we use the assumption that
  $\Phi$ is presaturated. We must show that if $N\colon\K\to\V$ lies
  in $\Phi^\dagger\K$, then $NF\colon\A\to\V$ lies in
  $\Phi^\dagger\A$. As a $\V$-functor, $N$ is a $\Phi$-colimit of
  representables. Restriction along $F$ preserves colimits, thus it
  will suffice to show that the restriction $\K(X,F-)$ of a
  representable $\K(X,-)$ is a $\Phi$-colimit of representables, but
  this is true by (1). 
\end{proof}

\begin{ese}
	Let $F\colon\A\to\V$ be a $\V$-functor.
	\begin{itemize}
		\item if $\Psi=\P$ and $\Phi=\emptyset$, then a $\Phi$-virtual left adjoint for $F$ is a left adjoint.
		\item if $\Psi=\emptyset$ and $\Phi=\P$, then a $\Phi$-virtual left adjoint for $F$ is a virtual left adjoint.
		\item if $\V=\bo{Set}$, $\Psi$ is the class of connected categories, and $\Phi$ that of the discrete ones, then a $\Phi$-virtual left adjoint is a left multiadjoint \cite{Die80:articolo}.
		\item if $\V=\bo{Set}$, $\Psi$ is the class of finite categories, and $\Phi$ that of the filtered ones, then a $\Phi$-virtual left adjoint is a left pluri-adjoint \cite{solian1990pluri}.
	\end{itemize}
\end{ese}

\begin{teo}\label{sound-strongcharact}
	Let $\Phi$ and $\Psi$ be classes of weights, $\Phi$ be a companion for $\Psi$, and $\A$ be a full subcategory of some $[\C,\V]$. The following are equivalent: \begin{enumerate}\setlength\itemsep{0.25em}
		\item $\A$ is accessible, accessibly embedded, and closed under $\Psi$-limits in $[\C,\V]$;
		\item $\A$ is accessibly embedded and $\Phi$-virtually reflective in $[\C,\V]$;
		\item $\A$ is a $\Phi$-virtual orthogonality class in $[\C,\V]$.
	\end{enumerate}
\end{teo}
\begin{proof}
	$(1)\Leftrightarrow (2)$ is a consequence of Theorem \ref{C-strong-charact}.
	
	$(2)\Rightarrow (3)$. Let $\K=[\C,\V]$ and denote by $L'$ the left $V_{\A}$-adjoint to $J$; in particular $Z_{\A}L'$ is a left adjoint to $J$ relative to $Z_{\A}V_{\A}$, which means that $\A$ is virtually reflective. Therefore we have:
	\begin{center}
		\begin{tikzpicture}[baseline=(current  bounding  box.south), scale=2]
			
			\node (a) at (2.2,0.8) {$\P^\dagger \A$};
			\node (b) at (2.2,0) {$\P^\dagger \K$};
			\node (c) at (1.1,0.8) {$\Phi^\dagger \A$};
			\node (d) at (1.1,0) {$\Phi^\dagger \K$};
			\node (e) at (2.28,0.4) {$\vdash$};
			\node (f) at (0,0.8) {$\A$};
			\node (g) at (0,0) {$\K$};
			
			\path[font=\scriptsize]
			
			(a) edge [bend left, <-] node [right] {$L$} (b)
			(a) edge [right hook->] node [left] {$R$} (b)
			(c) edge [right hook->] node [above] {$Z_{\A}$} (a)
			(d) edge [right hook->] node [below] {$Z_{\K}$} (b)
			(c) edge [right hook->] node [right] {$\Phi^\dagger J$}(d)
			(f) edge [right hook->] node [above] {$V_{\A}$} (c)
			(g) edge [right hook->] node [below] {$V_{\K}$} (d)
			(f) edge [right hook->] node [left] {$J$}(g)
			(g) edge [->] node [above] {$L'$}(c);
			
		\end{tikzpicture}	
	\end{center}
	where $Z_{\A}L'\cong LZ_{\K}V_{\K}$ since they are both left $(Z_{\A}V_{\A})$-adjoints to $J$.\\
	By the virtual case we know that $\A$ is the virtual orthogonality class defined by 
	$$ \M:=\{\eta_X\colon Z_{\K}V_{\K}X\to RLZ_{\K}V_{\K}X |\ X\in\K_{\alpha} \}. $$%
	for some $\alpha$, where $\eta$ is the unit of $L\dashv R$. We shall show that this is actually a $\Phi$-virtual orthogonality class. For each $X\in\K_{\alpha}$ we have 
	$$ RLZ_{\K}V_{\K}X\cong RZ_{\A}L'X\cong Z_{\K}(\Phi^\dagger J)L'X; $$%
	therefore $\M$ is contained in $\Phi^\dagger\K$ and coincides (up to isomorphism) with
	$$ \M':=\{\eta_X\colon V_{\K}X\to (\Phi^\dagger J)L'X |\ X\in\K_{\alpha} \}, $$%
	which exhibits $\A$ as a $\Phi$-virtual orthogonality class.
	
	$(3)\Rightarrow (1)$. $\A$ is accessible by \cite[Theorem~4.32]{LT22:virtual} because it is in particular a virtual orthogonality class. Let $\K=[\C,\V]$ and $\M$ be the set of arrows in $\Phi^\dagger \K$ defining $\A$; then $\A$ can be seen as the pullback
	\begin{center}
		\begin{tikzpicture}[baseline=(current  bounding  box.south), scale=2]
			
			\node (a0) at (0,0) {$\M^{\perp}$};
			\node (b0) at (0.9,0) {$\Phi^\dagger \K$};
			\node (c0) at (0,0.8) {$\A$};
			\node (d0) at (0.9,0.8) {$\K$};
			
			\path[font=\scriptsize]
			
			(a0) edge [right hook->] node [above] {} (b0)
			(c0) edge [right hook->] node [left] {} (a0)
			(d0) edge [right hook->] node [right] {$V_{\K}$} (b0)
			(c0) edge [right hook->] node [below] {} (d0);
			
		\end{tikzpicture}	
	\end{center}
	where $\M^{\perp}$ is the orthogonality class of $\Phi^\dagger \K$ in the usual sense. Now note that $\M^{\perp}$ is closed under all small limits that exist in $\Phi^\dagger \K$, and $V_{\K}$ preserves $\Psi$-limits since $\Phi$ is a companion for $\Psi$; it follows then that $\A$ is closed in $\K$ under $\Psi$-limits as well. \qedhere
	
\end{proof}

And for a general $\V$-category $\A$ we obtain the following.

\begin{teo}\label{sound-weakcharacterization}
	Let $\Phi$ and $\Psi$ be classes of weights for which $\Phi$ is a companion for $\Psi$. The following are equivalent for a $\V$-category $\A$:\begin{enumerate}\setlength\itemsep{0.25em}
		\item $\A$ is accessible with $\Psi$-limits;
		\item $\A$ is accessible and $\Phi^\dagger \A$ is cocomplete;
		\item $\A$ is accessible and $\Phi$-virtually cocomplete;
		\item $\A$ is accessibly embedded and $\Phi$-virtually reflective in some $[\C,\V]$;
		\item $\A$ is a $\Phi$-virtual orthogonality class in some $[\C,\V]$;
		\item $\A$ is the category of models of a limit/$\Phi$-colimit sketch.
	\end{enumerate}
\end{teo}
\begin{proof}
	Note that $\Phi$ is an accessible companion for $\Psi$ since every $\mt C^\Phi$ is either empty or the whole presheaf $\V$-category. Thus $(1)\Rightarrow (2)\Rightarrow (3)\Rightarrow (4)$ are a consequence of Theorem~\ref{relatice-C-charact} and $(4)\Rightarrow (5)\Rightarrow (1)$ follow from Theorem~\ref{sound-strongcharact} above. Finally, $(1)\Leftrightarrow (6)$ follows from Theorem~\ref{Psi-sketch} and Remark~\ref{Phi-colim-sketch}.
\end{proof}

\begin{obs}
  Although on the face of it one could define virtual orthogonality
  with respect to any companion, we have not done so since it does not seem to be possible to generalize 
  Theorems~\ref{sound-strongcharact}
  and~\ref{sound-weakcharacterization} to the setting of more general
  companions. For example if we consider the colimit type $\mt F$ associated
  to wide pullbacks, then $\mt F$-virtual orthogonality in a locally
  presentable category would coincide with multiorthogonality since
  $\mt F_1^\dagger \K=\tx{Fam}^\dagger \K$ whenever $\K$ is a category
  of presheaves (if $\K$ has a terminal object, all free groupoid
  actions based in $\K$ must be indexed on the discrete
  groupoids). This would be in contrast with the existence of locally
  polypresentable categories which are not locally
  multipresentable. 
\end{obs}

\begin{obs}
	Given a class of weights $\Phi$ denote by $\Phi^-$ the class of weights whose indexed limits commute in $\V$ with $\Phi$-colimits. As observed in \cite[Example~5.8]{KS05:articolo}, it may not be the case that $\Psi^{+-}$ is equal to the saturation $\Psi^*$ of $\Psi$, since $\Psi^{+-}$ always contains the absolute weights, but $\Psi^*$ may not. In fact that is all that can go wrong when $\Psi$ is weakly sound, as we now show. 
	
	Let $\Psi$ be a locally small weakly sound class. Then $\Psi^+$ is both a companion for $\Psi$ and for $\Psi^{+-}$; therefore by Theorem~\ref{sound-weakcharacterization} above it follows that an accessible $\V$-category $\A$ is $\Psi$-complete if and only if it is $\Psi^{+-}$-complete. When $\A$ is small, by \cite[Theorem~3.7]{Ten2022continuity}, this is saying that a small $\V$-category is Cauchy complete and $\Psi$-complete if and only if it is $\Psi^{+-}$-complete (remember that $\Psi^{+-}$ already contains the absolute weights). Similarly, a $\V$-functor from such a $\V$-category into $\V$ is $\Psi$-continuous if and only if it is $\Psi^{+-}$-continuous (since they both correspond to $\Psi^+$-colimits of representables). Thus if we denote by $\Q$ the class of absolute weights we obtain the equality
	$$ \Psi^{+-}=(\Psi\cup \Q)^* $$%
	expressing $\Psi^{+-}$ as the saturation of $\Psi$ together with $\Q$.
\end{obs}

\subsection{Wide Pullbacks}\label{widepbks}

In this section we let $\V=\bo{Set}$ and let $\Psi$ consist of the weights for wide pullbacks. The colimit type $\mt F$ we consider is that consisting of those groupoid-indexed diagrams in $\bo{Set}$ which induce a free action in the sense of Example~\ref{freegroupoids}. Equivalently, a groupoid indexed functor $G\colon \G\to\bo{Set}$ is in $\mt F$ if and only if, writing $\G$ as a coproduct of groups $(\G_i)_i$, for every non identity $g\in\G_i$ the function $G(g)$ has no fixed points.   

Given a category $\A$, the category $\mt F_1\A$ of Definition~\ref{C_1A} can be described as the full subcategory of $[\A\op,\bo{Set}]$ with objects those functors $F\colon\A\op\to\bo{Set}$ for which there exists a groupoid $\G$ and a diagram $H\colon\G\to\A$ such that $H$ is pointwise-free (that is, $\A(A,H-)$ is a free action for any $A\in\A$) and $F$ is the colimit of $YH\colon\G\to[\A\op,\bo{Set}]$; the representable functors are already included in such description (take $\G=1$ and $H$ pointing at the representing object). Equivalently, $\mt F_1\A$ is the closure under coproducts of the full subcategory of $[\A\op,\bo{Set}]$ spanned by the colimits of pointwise-free group actions on the representables. For further characterizations of $\mt F_1\A$, see Remark~\ref{poly-initial} below.

Recall the notion of polylimit in a category:

\begin{Def}{\cite[Definition~0.12]{Lam89:PhD}}
	Let $H\colon\C\to\A$ be a diagram in a category $\A$. A {\em polylimit} of $H$ is given by a family of cones $(c_i\colon\Delta A_i\to H)_{i\in I}$, with $A_i\in\A$, with the following property: for any cone $c\colon\Delta C\to H$ there exists a unique $i\in I$ and a map $f\colon C\to A_i$ in $\A$ such that $c=c_i\circ \Delta f$, moreover $f$ is unique up to unique automorphism of $A_i$.
\end{Def}

A more explicit description of $\mt F_1\A$ can be given as follows:

\begin{obs}\label{poly-initial}
	A functor $F\colon\A\op\to \bo{Set}$ lies in $\mt F_1\A$ if and only if its category of elements $\tx{El}(F)=\A/F$ has a poly-terminal family. Indeed, the objects and automorphism defining a poly-terminal family in $\tx{El}(F)$ form a groupoid which is a final subcategory of $\tx{El}(F)$, so that $F$ is the colimit of this groupoid action; moreover, such action is representably free by the polylimit property of the family (it can also be seen as a consequence of the arguments in Section~\ref{set-comp}). Conversely, let $F\in\mt F_1\A$ and consider the set of objects of $\A$ defining $F$ as the colimit of a free groupoid action between representables; then these determine a poly-terminal family in $\tx{El}(F)$; this is essentially the content of the proof of \cite[Proposition~4.2]{HT96qc:articolo} (see also Section~\ref{Hu-tholen}). For this reason, the name {\em polyrepresentable} was suggested in \cite{hebert1997syntactic} for presheaves lying in $\mt F_1\A$.
\end{obs}

\begin{obs}
	 An alternative characterization of polyrepresentable presheaves is possible \cite{Lam89:PhD} using the idea of strict generic objects. An object $G$ of a category $\C$ is said to be {\em strict generic} if $\C(G,f)$ is invertible, for all $f\colon A\to B$ in $\C$. The category $\C$ is said to have {\em enough strict generic objects} if there is a small set $\G$ of strict generic objects with the property that for any object $C$ in $\C$ there is a map $G\to C$ for some $G$ in $\G$. Then $\C$ has enough strict generic objects if and only if it has a polyinitial family. Thus $F\colon\A\op\to\bo{Set}$ is polyrepresentable if and only if $\tx{El}(F)\op$ has enough strict generic objects. 
\end{obs}

\begin{prop}\label{polycomplete}
	Let $\A$ be a category, $V\colon \A\hookrightarrow\mt F_1\A$ be the inclusion, and $H\colon \C\to\A$ be a diagram. Then $H$ has a polylimit in $\A$ if and only if $VH$ has a limit in $\mt F_1\A$. In particular then $\A$ has ($\alpha$-small) polylimits if and only if $\mt F_1\A$ has ($\alpha$-small) limits of representables.
\end{prop}
\begin{proof}
	(This is an adaptation of \cite[Proposition~3.4]{HT96qc:articolo}.) Assume first that $H$ has a polylimit $(\Delta A_i\to H)_{i\in I}$ in $\A$; then consider the groupoid indexed diagram $G\colon \textstyle\sum_i\tx{Aut}(A_i)\to\A$ given simply by the inclusions of the $A_i$'s together with their automorphisms. The polylimit property implies that $\A(A,G-)$ lies in $\mt F$ for any $A\in\A$: if $f\colon A\to A_i$ and $g\in\tx{Aut}(A_i)$ satisfy $gf=f$ then, since $f$ corresponds to a cone $\Delta A\to H$, such a $g$ must be unique, but the identity also satisfies the equality; thus $g=1_{A_i}$. 
	
	As a consequence the colimit $X$ of $VG$ is an object of $\mt F_1\A$. Moreover the maps $(c_i\colon\Delta A_i\to H)_{i\in I}$ define a cocone out of $VG$ in $\mt F_1\A$ which in turn induces a map $c\colon\Delta X\to VH$. It's now easy to see that, by the polylimit properties of the $A_i$'s, the map $c$ exhibits $X$ as the limit of $VH$ in $\mt F_1\A$.
	
	Conversely, let $X$ be the limit of $VH$ in $\mt F_1\A$; then we can write $X$ as the colimit of $VG$ where $G\colon \textstyle\sum_i\G_i\to\A$ is representably in $\mt F$ and each $\G_i$ is a group. Let $A_i$ be the image in $\A$ of each group component $\G_i$; then the limiting cone of $X$ induces a family $(\Delta A_i\to H)_{i\in I}$ of cones over $H$. We prove that these exhibit $(A_i)_i$ as the polylimit of $\A$. To give a cone for $A\in\A$ over $H$ is the same as giving an arrow $VA\to X$; note now that $\mt F_1\A(VA,-)$ preserves the colimit of $VG$ defining $X$, and therefore $\mt F_1\A(VA,X)\cong \tx{colim}\A(A,G-)$ in $\bo{Set}$. It follows that giving an arrow $VA\to X$ is the same as giving a map $A\to A_i$, for a unique $i$, determined up to composition with some $G(g)\colon A_i\to A_i$; finally this $g$ is also unique because $\A(A,G-)$ is a free groupoid action by hypothesis. It follows that the family $(A_i)_i$ is the polylimit of $H$ in $\A$.
\end{proof}

\begin{obs}
	The dual of this result implies that the polycolimit of a diagram $H\colon\C\to\A$ exists in $\A$ if and only if the $\mt F$-virtual colimit of $H$ exists, and in that case they coincide. Thus, a category is polycocomplete if and only if it is $\mt F$-virtually cocomplete.
\end{obs}

Recall the following lemma of Lamarche:

\begin{lema}[Lemma~0.13 of \cite{Lam89:PhD}]\label{Lamarche}
	Let $\B$ be a category with wide pushouts. Then $\B$ has a polyterminal object if and only if it has a weakly terminal family.
\end{lema}

Thanks to this and Proposition~\ref{companion-terminal} we can easily prove the following:

\begin{prop}
	The class $\mt F$ is an accessible companion for the class of wide pullbacks.
\end{prop}
\begin{proof}
	Accessibility of $\mt F$ is given by Example~\ref{acccomp}, which also shows that the limit specifications in the sketches defining $\mt F$ are all connected. Thus we can use Proposition~\ref{companion-terminal} to show that $\mt F$ is a companion for the class of wide pullbacks. That $\mt F$ is compatible with wide pullbacks is given by \cite[Proposition~1.4]{HT96qc:articolo} --- see Section~\ref{Hu-tholen} for a comparison of our work with that in \cite{HT96qc:articolo}. Thanks to Proposition~\ref{polycomplete}, to show property (ii) of Proposition~\ref{companion-terminal}, we need to prove that every virtually cocomplete category $\B$ with wide pullbacks has a polyinitial object. This follows at once by the dual of Lemma~\ref{Lamarche} since every virtually cocomplete category $\B$ has a weakly initial family. Indeed, by virtual cocompleteness, the functor $\Delta 1\colon\B\to\bo{Set}$ is small; thus it is the left Kan extension of its restriction to a small full subcategory $\C$ of $\B$. The elements of $\C$ then form a weakly initial family in $\B$.
\end{proof}

\begin{Def}
	We say that a functor $F\colon\A\to\B$ has a {\em left
          polyadjoint} if it has an $\mt F$-virtual left adjoint; if $F$ is fully faithful we say that $\A$ is {\em polyreflective} in $\B$.
\end{Def}

Traditionally, one says that $F\colon\A\to\B$ has a left polyadjoint if, for any $B\in\B$, the category $B/F=\tx{El}(\B(B,F-))\op$ has a polyinitial object \cite[Page~35]{Lam89:PhD}. By Remark~\ref{poly-initial}, this is equivalent to saying that $\B(B,F-)$ lies in  $\mt F_1^\dagger\A$.  Thus our definition coincides with the classical notion of left polyadjoint. 

It follows that a fully faithful inclusion $J\colon\A\to\K$  is polyreflective if and only if for any $X\in\K$ the slice $X/\A$ has a polyinitial object.

\begin{obs}
	In \cite[Section~1.2]{taylor1990trace} Taylor shows that $F\colon\A\to\B$ has a left polyadjoint if and only if the induced functor $F/A\colon\A/A\to\B/FA$ has a left adjoint for any $A\in\A$.
\end{obs}

Now, since $\mt F$ is a companion for the class of wide pullbacks and thanks to Theorem~\ref{C-strong-charact} we obtain:

\begin{teo}
	Let $\K$ be an accessible category with wide pullbacks and $\A$ a full subcategory of $\K$. The following are equivalent:\begin{enumerate}\setlength\itemsep{0.25em}
		\item $\A$ is accessible, accessibly embedded, and closed under wide pullbacks in $\K$;
		\item $\A$ is accessibly embedded and polyreflective in $\K$.
	\end{enumerate}
\end{teo}

Regarding sketches, note that, by the presentation of $\mt F$ as in Example~\ref{galoisian}, the categories of $\mt F$-models of sketches are the same as the categories of models of galoisian sketches \cite{Age92:articolo}. Thus, thanks to Theorem~\ref{relatice-C-charact} and the dual of Proposition~\ref{polycomplete}, we obtain:

\begin{teo}
	Let $\A$ be a category; the following are equivalent: \begin{enumerate}\setlength\itemsep{0.25em}
		\item $\A$ is accessible with wide pullbacks;
		\item $\A$ is accessible and $\mt{F}_1^\dagger \A$ is cocomplete;
		\item $\A$ is accessible and polycocomplete;
		\item $\A$ is accessibly embedded and polyreflective in $[\C,\bo{Set}]$ for some $\C$;
		\item $\A$ is the category of models of a galoisian sketch.
	\end{enumerate}
\end{teo}

The equivalence of $(1)$ and $(3)$ was first given by Lamarche in \cite[Theorem~0.20]{Lam89:PhD}, then Ageron further added condition $(5)$ in \cite[Theorem~4.19]{Age92:articolo}.

\subsubsection{Quasi-coproducts}\label{Hu-tholen}

The notion appearing below has been used in \cite{HT96qc:articolo}:

\begin{Def}\label{qc}
	A groupoid indexed diagram $H\colon \G\to\B$ is called {\em quasi-discrete} if for each non-initial $B\in\B$ the functor $\B(B,H-)$ is a free action in $\bo{Set}$; that is, if $\B(B,H-)\in\mt F$ for any non-initial $B$. A {\em quasi-coproduct} is the colimit of such a diagram.
\end{Def}

Given a category $\B$ with quasi-coproducts, Hu and Tholen define $\B_q$ to be the full subcategory of $\B$ consisting of those objects $B$ for which $\B(B,-)$ preserves quasi-coproducts; they call $\B$ {\em quasi-based} if each object is a quasi-coproduct of objects from $\B_q$. 

The following shows that our notion of diagram representably in $\mt F$ is comparable with that of quasi-discrete diagram.

\begin{prop}
	Let $\A$ be a category, $V\colon \A\hookrightarrow\mt F_1\A$ be the inclusion, and $H\colon \G\to\A$ be a groupoid indexed diagram. Then $\A(A,H-)\in\mt F$ for any $A$ in $\A$ if and only if $VH$ is quasi-discrete in $\mt F_1\A$.
	
	Similarly, let $\B$ be a quasi-based category and $H\colon \G\to\B_q$ be a groupoid indexed diagram. Then $H$ is quasi-discrete in $\B$ if and only if $\B_q(B,H-)\in\mt F$ for any $B$ in $\B_q$.
\end{prop}
\begin{proof}
	Assume that $\A(A,H-)\in\mt F$ for any $A$ and that there exists a non-initial object $X\in \mt F_1\A$ such that $\mt F_1\A(X,VH-)$ is not free. Since $X$ is not initial and is a colimit of  elements from $\A$, there exists a map $VA\to X$ for some $A$ in $\A$. It then follows that $\A(A,H-)\cong\mt F_1\A(VA,VH-)$ is not free as well, leading to a contradiction. Conversely, assume that $VH$ is quasi-discrete in $\mt F_1\A$. Since the terminal object of $\mt F_1\A$ is computed as in $[\A\op,\bo{Set}]$ it can't lie in $\A$; thus $\A(A,H-)$ is in $\mt F$ for any $A\in\A$ by definition.
	
	The same proof applies to the second statement since the initial object $0$ of $\B$ doesn't lie in $\B_q$; indeed $\B(0,-)$ doesn't preserve coproducts, and hence can't preserve quasi-coproducts. 
\end{proof}

By the proposition above, it follows that a category $\B$ with quasi-coproducts is quasi-based if and only if $\B\simeq\mt F_1\B_q$. Thus, the content of \cite[Proposition~3.4]{HT96qc:articolo} coincides with that of our Proposition~\ref{polycomplete}.

The difference between our approach and that of Hu and Tholen is simply that, while in \cite{HT96qc:articolo} they are interested in recognising those categories that arise as free cocompletions under colimits of free groupoid actions, we want to construct such free cocompletions starting from a given category. In fact, when freely adding colimits of free groupoid actions to a category $\A$, one needs to consider those diagrams $H\colon\G\to\A$ that lie representably in $\mt F$; however, when determining if a category $\B$ is a free completion under colimits of free groupoid actions the right notion to consider is that of quasi-discrete diagram.

\subsection{Products and powers by a dense generator}\label{products+G-powers}

Let $\V$ be symmetric monoidal closed and locally presentable as
usual, and $\G$ be a (possibly large) dense generator of $\V_0$
containing the unit and closed under tensor product, and
satisfying Assumption~\ref{ass:E} below. 

\begin{Def}\label{E-maps}
	Let $\E\subseteq\V^\mathbbm{2}$ be the class of maps $e$ for which $G\pitchfork e$ is a regular epimorphism for any $G\in\G$. In particular every map in $\E$ is a regular epimorphism.
\end{Def}

\begin{as}\label{ass:E}
  We suppose that at least one of the following conditions holds:\begin{itemize}
	\item[(I)] the unit $I$ is regular projective;
	\item[(II)] if $f\circ g$ is a regular epimorphism in $\V$ then so is $f$, and $\E$ is closed under products in $\V^\mathbbm{2}$.
\end{itemize} 
\end{as}

\begin{lema}\label{pointwise}
Suppose that condition (I) holds. Then $e\in\E$ if and only if $\V_0(P,e)$ is surjective for any $P\in\G$. In particular $\E$ is closed under products and under composition.
\end{lema}
\begin{proof}
	If $e\in\E$ and $P\in\G$ then $\V_0(P,e)\cong\V_0(I,P\pitchfork e)$ is surjective since $I$ is regular projective in $\V_0$ and $P\pitchfork e$ is a regular epimorphism.
	Conversely, assume that $\V_0(P,e)$ is surjective for any $P\in\G$. Let $H$ be the inclusion of $\G$ in $\V_0$; by hypothesis we have a fully faithful $J=\V_0(H,1)\colon\V_0\hookrightarrow \P\G$ which has a left adjoint $L$ since $\V_0$ is cocomplete. Then our hypothesis is saying that $Je$ is a regular epimorphism in $\P\G$: the kernel pair of $Je$ exists in $\P\G$ since it's the image through $J$ of the kernel pair of $e$ in $\V_0$; hence $Je$, being  by definition a pointwise surjection, is the coequalizer of its kernel pair. Thus $e\cong LJe$ is a regular epimorphism in $\V_0$ by cocontinuity of $L$. Now, given $G\in\G$, the morphism $G\pitchfork e$ still satisfies that $\V_0(P,G\pitchfork e)$ is surjective for any $P\in\G$ (since $\G$ is closed under tensor product); thus $G\pitchfork e$ is a regular epimorphism by the previous argument. It follows that $e\in\E$.
\end{proof}

\begin{Def}\label{D-E-maps}
	We say that a pair $f,g\colon X\to Y$ in $\V$ is a {\em
          $\G$-pseudo equivalence relation} if it factors as a map
        $e\colon X\twoheadrightarrow Z$ in $\E$ followed by the
        kernel pair $h,k\colon Z\to Y$ of some map in
        $\E$, necessarily the coequalizer of $h$ and $k$ (and so also
        of $f$ and $g$).
	Denote by $\mt C$ the colimit type generated by the $\G$-pseudo equivalence relations: $\mt C_M$ is non-empty only for $M=\Delta I:\C\op\to\V$, where $\C$ is the free $\V$-category on a pair of arrows, and in that case $\mt C_{\Delta I}$ is the full subcategory of $[\C,\V]$ spanned by the $\G$-pseudo equivalence relations in $\V$.
\end{Def}

\begin{obs}
	A $\V$-functor $F\colon\A\to\V$ lies in $\mt C^\dagger_1 \A$ if and only if there exists a fork
	\begin{center}
		
		\begin{tikzpicture}[baseline=(current  bounding  box.south), scale=2]
			
			\node (b) at (0.6,0) {$\A(B,-)$};
			\node (c) at (2,0) {$\A(A,-)$};
			\node (d) at (3,0) {$F$};

			\path[font=\scriptsize]
			
			([yshift=1.5pt]b.east) edge [->] node [above] {} ([yshift=1.5pt]c.west)
			([yshift=-1.5pt]b.east) edge [->] node [below] {} ([yshift=-1.5pt]c.west)
			(c) edge [->] node [above] {$q$} (d);
		\end{tikzpicture}
	\end{center}
	where both $q\colon\A(A,-)\to F$ and the induced map $\A(B,-)\to P$, into the kernel pair $P$ of $q$, are pointwise in $\E$.
\end{obs}

\begin{prop}\label{D-companion}
	The class $\mt C$ is a companion for the class of products and $\G$-powers. If $\G$ is small then $\mt C$ is an accessible companion.
\end{prop}
\begin{proof}
	That $\mt C$ is compatible with products and $\G$-powers in $\V$ is a consequence of the fact that the maps in $\E$ are stable under them, and kernel pairs commute with any limit. 
	
	Assume now that $\A$ is virtually cocomplete with products and $\G$-powers, and consider $F\colon \A\to\V$ to be a small functor which preserves these limits. Let $Y\colon \A\op\to\P(\A\op)$ be the Yoneda embedding; then, by smallness of $F$ and since $\G$ is a dense generator, there is a regular epimorphism  
	$$q\colon \sum_i(P_i\cdot YA_i)\twoheadrightarrow F$$%
	with $P_i\in\G$ for any $i$. Consider now $A=\textstyle\prod_i(P_i\pitchfork A_i)$ in $\A$ and the comparison $\textstyle\sum_i(P_i\cdot YA_i)\to YA$. Since $F$ preserves products and $\G$-powers, $q$ factorizes through the comparison via a map
	$$e\colon YA\longrightarrow F.$$%
	We wish to prove that $e$ lies pointwise in $\E$. The proof will depend on which condition, (I) or (II), holds in $\V$. Assume that (I) holds; since both $YA$ and $F$ preserve $\G$-powers, we have an isomorphism $G\pitchfork e_B\cong e_{G\pitchfork B}$ for any $G\in\G$ and $B\in\A$. Therefore
	$$ \V_0(G,e_B)\cong\V_0(I,G\pitchfork e_B)\cong \V_0(I,e_{G\pitchfork B}) $$%
which is surjective since $\V_0(q_B)$ is so by condition
(I); thus $e_B\in\E$ for any $B$. On the other hand, if (II) holds
then $e_B$ is a regular epimorphism for any $B$ in $\A$ (since $q_B$ is one), and thus also $G\pitchfork e_B\cong e_{G\pitchfork
  B}$ is a regular epimorphism, thus once again $e_B\in\E$.
	
	Now, since $\A$ is virtually cocomplete, $\P(\A\op)$ is complete, and the kernel pair $K$ of $e$ is still small and preserves the same limits as $F$; hence by the same arguments we can find a map $e'\colon YA'\to K$ which lies pointwise in $\E$. It follows that $F$ can be expressed as the coequalizer of a $\G$-pseudo equivalence relation between representables; in other words it lies in $\mt C_1(\A\op)$. This proves that $\mt C$ is a companion for products and $\G$-powers.
	
	Assume now that $\G$ is small. We define a sketch for $\mt C_{\Delta I}$ in the spirit of Example~\ref{acccomp}(3). Consider the $\V$-category $\C'$ in $\P\C$ spanned by: $\C$, the coequalizer $q$ of $(f,g)$, the kernel  pair $h,k\colon Z\to Y$ of $q$, the kernel pair $(h',k')$ of the map $e\colon X\to Z$ induced by the kernel pair $(h,k)$, and $\G$-powers of all these. Then define the sketch on $\C'$ with limit conditions $\mathbb{L}$ specifying $(h,k)$ and $(h',k')$ as the kernel pairs of $q$ and $e$ respectively. The only colimit conditions $\mathbb{C}$ specify, for any $G\in\G$, the maps $G\pitchfork q$ and $G\pitchfork e$ as the coequalizers of $(G\pitchfork h,G\pitchfork k)$ and $(G\pitchfork h',G\pitchfork k')$ respectively. Since these conditions force $q$ and $e$ to be identified with maps in $\E$, it's then easy to see that 
	$$\mt C_{\Delta I}\simeq \tx{Mod}(\C',\mathbb{L},\mathbb{C}).$$%
	As a consequence it is accessible and closed under filtered colimits, products, and $\G$-powers in $[\C,\V]$. 
\end{proof}

We can now apply Theorem~\ref{C-strong-charact} to obtain: 

\begin{teo}
	Let $\K$ be an accessible $\V$-category with products and $\G$-powers, and let $\A$ be a full subcategory of $\K$. The following are equivalent:\begin{enumerate}\setlength\itemsep{0.25em}
		\item $\A$ is accessible, accessibly embedded, and closed under products and $\G$-powers;
		\item $\A$ is accessibly embedded and $\mt C$-virtually reflective.
	\end{enumerate}
\end{teo}

Regarding sketches, note that, by the presentation of $\mt C$ as in
Proposition~\ref{D-companion} and Section~\ref{D-sketch}, the
$\V$-categories of $\mt C$-models of sketches are the same as the
$\V$-categories of models of limit sketches where in addition some
maps are specified to lie in $\E$; such sketches are
commonly called limit/$\E$ sketches. Thus, in light of
Proposition~\ref{D-companion}, Theorems~\ref{relatice-C-charact} and \ref{Psi-sketch} in this case become:

\begin{teo}\label{prod+G-powers-weak}
	The following are equivalent for a $\V$-category $\A$: \begin{enumerate}\setlength\itemsep{0.25em}
		\item $\A$ is accessible with products and $\G$-powers;
		\item $\A$ is accessible and $\mt{C}_1^\dagger \A$ is cocomplete;
		\item $\A$ is accessible and $\mt C$-virtually cocomplete;
		\item $\A$ is accessibly embedded and $\mt C$-virtually reflective in $[\C,\V]$ for some $\C$.
	\end{enumerate}
	If $\G$ is small, then they are further equivalent to: \begin{enumerate}
	\item[(5)] $\A$ is the $\V$-category of models of a limit/$\E$ sketch.
	\end{enumerate}
\end{teo}

For some of the examples outlined below, this theorem relates to the results of \cite{LR12:articolo}, although at this point, it doesn't capture them completely. We deal with this in Section~\ref{wrsnc}: see Theorem~\ref{weak-prod+pow}.

\begin{es}[\em Products and small powers]
	Assume that $\V$ satisfies (I) and consider $\G=\V_0$ as the dense generator. Then, using Lemma~\ref{pointwise}, it's easy to see that the class $\E$ consists exactly of the split epimorphisms in $\V$.
	
	Then the theorem above classifies accessible $\V$-categories with products and powers. A fully faithful inclusion $J\colon\A\to\K$ is $\mt C$-virtually reflective in this context if for any $X\in\K$ there is a pointwise-split map $q\colon\A(A,-)\to \K(X,J-)$ whose kernel pair is also a pointwise-split subobject of some representable $\A(B,-)$.
\end{es}

\begin{es}[\em Products and projective powers]\label{products}
	Let $\V$ be a symmetric monoidal quasivariety as in \cite{LT20:articolo}; in this section we consider $\G$ to consist of the enriched finitely presentable and regular projective objects of $\V$. Then $\V$ satisfies condition (II) being regular. The class $\E$ is simply given by the regular epimorphisms in $\V$. The corresponding colimit type $\mt R$, like in the ordinary case, is the one formed by the pseudo-equivalence relations in $\V$: we say that a pair $f,g\colon X\to Y$ in $\V$ is a {\em pseudo-equivalence relation} if it factors as a regular epimorphism $e\colon X\twoheadrightarrow Z$ followed by a kernel pair $h,k\colon Z\to Y$. 
	
	\noindent Note that a Cauchy complete $\V$-category has products and $\G$-powers if and only if it has products and powers by projective objects. This is because every projective object of $\V$ is a split subobject of coproducts of elements of $\G$ \cite[Proposition~4.8]{LT20:articolo}. Therefore Theorem~\ref{prod+G-powers-weak} provides a characterization of the accessible $\V$-categories with products and projective powers. 
	
	\noindent A fully faithful inclusion $J\colon\A\to\K$ is $\mt C$-virtually reflective in this context if for any $X\in\K$ there is a regular epimorphism $q\colon\A(A,-)\to \K(X,J-)$ whose kernel pair is also a regular quotient of some representable $\A(B,-)$. In the ordinary case, this means that there exist $r\colon X\to JA$ in $\K$ and $h_1,h_2\colon A\to B$ in $\A$ such that every map $f\colon X\to JC$ factors as $f=Jg\circ r$ for some $g\colon A\to C$ in $\A$ (that is, $r$ is a weak reflection of $X$ into $\A$) and for any two $g_1,g_2$ with $Jg_2\circ r=Jg_2\circ r$ there exists $k\colon B\to C$ such that $g_1=k\circ h_1$ and $g_2=k\circ h_2$.
	
	\noindent What this gives in the case $\V=\bo{Set}$
    is not exactly the traditional characterization of \cite[Chapter~4]{AR94:libro} which uses weak colimits and weak left adjoints. See Section \ref{wrsnc} for the relation between left $\mt R$-adjoints and weak left adjoints, and between colimits in $\mt R_1^\dagger\A$ and weak colimits.
\end{es}

\begin{es}[\em Products and powers by $\mathbbm{2}$]\label{prod+powers2}
	Let $\V=\bo{Cat}$, which satisfies condition (I), and consider $\G=\{\mathbbm{2}^n\}_{n\in\mathbb{N}}$ together with the induced class $\E$.
	
	\noindent By Lemma~\ref{pointwise}, $f\colon\C\to\D$ in $\bo{Cat}$ lies in $\E$ if and only if it is {\em surjective on cubes}; that is, if $\mathbbm{2}^n\pitchfork f$ is surjective on objects for any $n\in\mathbb{N}$. The colimit type $\mt H$ induced by $\E$ can be described as the one formed by the pairs $f,g\colon \C\to \D$ in $\bo{Cat}$ which factor as a surjective on cubes functor $e\colon \C\to \E$ followed by a kernel pair $h,k\colon \E\to \D$ whose coequalizer is surjective on cubes.
	
	\noindent Given the ordinal $\mathbbm{n}:=\{0\to1\dots n-1\}$, one could try to consider those morphisms $e$ such that $\mathbbm{n}\pitchfork e$ is a regular epimorphism for each $n\geq 0$. Since $\mathbbm{n}$ is a split subobject of $\mathbbm{2}^{n-1}$, it follows that every cube epimorphism also satisfies this property. However the converse doesn't hold: consider the inclusion of the non-commutative square into the commutative square, this satisfies the condition on ordinals in $\bo{Cat}$ but is not surjective on cubes. Note that, by \cite[Proposition~6.2]{pedicchio_tholen_2003} and the observation just made, every surjective on cubes functor in $\bo{Cat}$ is an effective descent morphism.
\end{es}

\begin{es}[\em Products and finite powers]\label{pure}
	Let $\V$ be locally finitely presentable as a closed category with a regular projective unit, so that condition (I) is satisfied. Consider $\G=\V_f$, then the maps in $\E$ are usually called {\em pure epimorphisms}: these are the morphisms $e$ for which $\V_0(A,e)$ is surjective for any $A\in\V_f$. 	
\end{es}

\subsection{Flexible Limits}\label{flexible}

For this section we let $\V=\bo{Cat}$ and consider the class of weights $\bo{Flex}$ for flexible limits; these are generated by products, inserters, equifiers, and splittings of idempotents. See \cite{lack20102} as a reference for 2-limits.

\begin{obs}
	Note that, when dealing with accessible categories, it's equivalent to consider $\bo{Flex}$ or the class $\bo{PIE}$ of PIE-limits. This is because every accessible 2-category is Cauchy complete, and flexible limits are generated by PIE-limits and splittings of idempotents.
\end{obs}

Before describing a companion for $\bo{Flex}$ let's recall some facts about retract equivalences, coisoidentifiers, and a notion of kernel in the 2-categorical context.

A retract equivalence in a 2-category $\K$ is a morphism $q\colon D\to E$ for which there exist a section $s\colon E\to D$, so that $qs=1_E$, and an invertible 2-cell $\sigma\colon sq\cong 1_D$. In $\bo{Cat}$ retract equivalences are precisely those equivalences that are moreover surjective on objects.

We shall present retract equivalences in $\bo{Cat}$ as part of a kernel-quotient system (\cite{BG14:articolo}, see also Section~\ref{wrsnc}). The {\em kernel} of this system is given by what we call an isokernel cell. 

\begin{Def}
	Let $q\colon D\to E$ be a morphism in a 2-category $\K$; the {\em isokernel cell} of $q$ is the universal invertible 2-cell 
	\begin{center}
		\begin{tikzpicture}[baseline=(current  bounding  box.south), scale=2]
			
			\node (a) at (0,0) {$ C$};
			\node (b) at (1,0) {$ D$};
			\node (c) at (0.5,0) {$\ \ \Downarrow\phi$};

			\path[font=\scriptsize]
			
			(a) edge[bend left, ->] node [above] {$\pi_1$} (b)
			(a) edge[bend right, ->] node [below] {$\pi_2$} (b);
			
		\end{tikzpicture}
	\end{center}
	such that $q\phi=id$.
\end{Def}

In $\bo{Cat}$, the category $C$ is given by the full subcategory of $D^{\mathbbm{2}}$ whose objects are the isomorphisms $f\colon x\to y$ of $D$ for which $q(f)=id$. Then $\pi_1$ and $\pi_2$ are the domain and codomain projections, and $\phi_f=f\colon \pi_1(f)\to\pi_2(f)$. 

The {\em quotient} of the system is given by coisoidentifiers:

\begin{Def}
	Let $\phi\colon \pi_1\Rightarrow \pi_2\colon C\to D$ be an invertible 2-cell in a 2-category $\K$. The {\em coisoidentifier} $q\colon  D\to E$ of $\phi$ is the universal morphism out of $D$ satisfying the equality $q\phi=\tx{id}$.
\end{Def}

See Section~\ref{wrsnc} for a description in terms of weighted colimits.

The next proposition shows that retract equivalences are the (split) colimit of their isokernel cells. 

\begin{prop}
	Every retract equivalence in a 2-category $\K$ is the coisoidentifier of its isokernel cell. 
\end{prop}
\begin{proof}
	Let $q\colon D\to E$ be a retract equivalence in $\K$; then we can take a section $s\colon E\to D$  and an invertible 2-cell $\sigma\colon sq\cong 1_D$.
	Given the isokernel cell $\phi\colon \pi_1\Rightarrow \pi_2\colon C\to D$ of $q$, by the universal property of the limit applied to $\sigma$ there exists $v\colon  D\to C$ with $\pi_1 v=sq$, $\pi_2 v=1_D$, and $\phi v=\sigma$. It now follows easily that $q$ is the (split) coisoidentifier of $\phi$: given any map $h\colon D\to F$ such that $h\phi=\tx{id} $, then $(hs)q=h\pi_1 v=h\pi_2 v=h$ so that $h$ factors through $q$. The factorization is unique since $q$ is an epimorphism.
\end{proof}

In general, it is not true that the coisoidentifier of any isokernel cell is a retract equivalence. For instance, in $\bo{Cat}$, if $\phi\colon \pi_1\Rightarrow \pi_2\colon C\to D$ is the isokernel cell of some map $r\colon D\to F$, and $D$ has a non-identity isomorphism $f\colon x\to x$ such that $r(f)=id$, then the colimit $q\colon D\to E$ of $\phi$ will send $f$ to the identity map. Thus $q$ is not faithful and in particular not a retract equivalence.

We can avoid this problem by introducing the notion of acyclic isokernel cell: 

\begin{Def}\label{equivalence2rel}
	An isokernel cell $\phi\colon \pi_1\Rightarrow \pi_2\colon C\to D$ in $\bo{Cat}$ is called {\em acyclic} if $\phi_c=id$ whenever $\pi_1c=\pi_2c$. 
\end{Def}

\begin{obs}\label{alsoacyclic}
	Let $k\colon B\to C$ be the equalizer of $\pi_1$ and $\pi_2$, then $\phi$ is acyclic if and only if  $\phi k=id$. Equivalently, $\phi$ is acyclic if and only if the equalizer $k$ of $\pi_1$ and $\pi_2$ coincides with the identifier of $\phi$.
\end{obs}

\begin{prop}\label{catexact}
	An isokernel cell $\phi\colon \pi_1\Rightarrow \pi_2\colon C\to D$ in $\bo{Cat}$ is acyclic if and only if its coisoidentifier is a retract equivalence. In this case, $\phi$ is also the isokernel cell of its coisoidentifier.
\end{prop}
\begin{proof}
	Assume first that the coisoidentifier $q\colon D\to E$ of $\phi$ is a retract equivalence. By the property of the colimit, every morphism in $\D$ of the form $\phi_c\colon d\to d$ is sent to the identity morphism by $q$. But $q$ is an equivalence; thus $\phi_c=1_d$ and $\phi$ is acyclic. 
	
	Conversely let's assume that $\phi$, as above, is an acyclic isokernel cell in $\bo{Cat}$. We shall give an explicit construction of its coisoidentifier. 
	
	Consider the following equivalence relation on the objects of
        $D$: two objects $d$ and $e$ of $\D$ are related if and only
        if there exists $c$ in $\C$ for which $\phi_c\colon d\to
        e$; 
	in other words,  if $\pi_1c=d$ and $\pi_2c=e$.
	This is actually an equivalence relation on the objects of $ D$ since $\phi$ is an isokernel cell (the inverse or composition of any maps of the form $\phi_c$ is still of the form $\phi_{c'}$ for some $c'\in C$). Now, for each equivalence class of objects choose a representative. Let $E$ be the full subcategory of $D$ consisting of the chosen representatives, and $s\colon E\to D$ be the inclusion; clearly this is essentially surjective on objects and so an equivalence.
	
	For each $d\in D$, let $qd\in E$ be the chosen representative of the equivalence class of $d$, and let $\sigma_d\colon sqd\to d$ be the unique isomorphism of the form $\phi_c$ for some unique $c\in C$ (the uniqueness of $\sigma_d$ follows from the fact that $\phi$ is acyclic: given any other $\sigma_d'$ the composite $\sigma_d^{-1}\sigma_d'$ must be the identity map). Then $q$ defines a functor $q\colon D\to E$ in such a way that $qs=1_E$ and the $\sigma_d$ define a natural isomorphism $\sigma\colon sq\cong 1_D$. It's now easy to see that $\phi$ is also the isokernel cell of $q$ and thus, by the universal property of the limit, there is a unique $v\colon D\to C$ with $\pi_1 v=1_D$, $\pi_2 v=sq$, and $\phi v=\sigma$. Now $q$ is a (split) coisoidentifier of $\phi$.
\end{proof}

We will now construct a companion $\mt P$ for $\bo{Flex}$ by considering retract equivalences and acyclic isokernel cells in $\bo{Cat}$. 

\begin{Def}\label{2-companion}
	Let $\mt P$ be the colimit type given by: $\mt P_M $ is non empty only when $ M =\Delta 1\colon \W\op\to \bo{Cat}$, where $\W=\{\cdot\cong\cdot\}$ is the 2-category freely generated by an invertible 2-cell. In that case $\mt P_M $ consists of the (invertible) 2-cells
	\begin{center}
		\begin{tikzpicture}[baseline=(current  bounding  box.south), scale=2]
			
			\node (a) at (0,0) {$X$};
			\node (b) at (1,0) {$Y$};
			\node (c) at (0.5,0) {$\ \ \Downarrow\psi$};

			\path[font=\scriptsize]
			
			(a) edge[bend left, ->] node [above] {$f$} (b)
			(a) edge[bend right, ->] node [below] {$g$} (b);
			
		\end{tikzpicture}
	\end{center}
	in $\bo{Cat}$ which factor as a retract equivalence $e\colon X\to Z$ followed by an acyclic isokernel cell $\phi\colon \pi_1\Rightarrow \pi_2\colon Z\to Y$ (so that $f=\pi_1e$, $g=\pi_2e$, and $\psi=\phi e$).
\end{Def}

\begin{obs}\label{alsoP_M}
	If $\psi$ in $\mt P_M$ as above, since $e$ is in particular surjective on objects the coisoidentifier of $\psi$ and $\phi$ coincide and is a retract equivalence by Proposition~\ref{catexact} above. Moreover $\phi$ is then the isokernel cell of such a coisoidentifier, and $e$ is the map induced by the universal property of the limit.
\end{obs}

\begin{obs}
	A 2-functor $F\colon\A\to\bo{Cat}$ lies in $\mt P^\dagger_1 \A$ if and only if there exists a diagram
	\begin{center}
		
		\begin{tikzpicture}[baseline=(current  bounding  box.south), scale=2]
			
			\node (b) at (0.5,0) {$\A(B,-)$};
			\node (c) at (2,0) {$\A(A,-)$};
			\node (c1) at (1.2,0) {$\ \ \Downarrow\psi$};
			\node (d) at (3,0) {$F$};

			\path[font=\scriptsize]
			
			([yshift=2.5pt]b.east) edge [bend left, ->] node [above] {} ([yshift=2.5pt]c.west)
			([yshift=-2.5pt]b.east) edge [bend right, ->] node [below] {} ([yshift=-2.5pt]c.west)
			(c) edge [->] node [above] {$q$} (d);
		\end{tikzpicture}
	\end{center}
	where $\psi$ is invertible, $q\psi$ is an identity, and both $q\colon\A(A,-)\to F$ and the induced map $\A(B,-)\to P$, into the isokernel cell $P$ of $q$, are pointwise retract equivalences.
\end{obs}

In the next proposition we will use \cite[Theorem~6.2]{BLV:articolo}, which states that if $\A$ is a 2-category with flexible limits and $F\colon \A\to\bo{Cat}$ is a 2-functor which preserves them and (whose underlying functor) satisfies the solution-set condition, then there exists $A\in\A$ and a pointwise retract equivalence $q\colon \A(A,-)\to F$. 

\begin{prop}
	The colimit type $\mt P$ is an accessible companion for $\bo{Flex}$.
\end{prop}
\begin{proof}
	Let us first show that $\mt P$ is compatible with $\bo{Flex}$; for that we need to prove that the inclusion $ J\colon\mt P_M\hookrightarrow[\W,\bo{Cat}]$ (with $M$ and $\W$ as in the definition above) and the composite of $J$ with the colimit 2-functor $[\W,\bo{Cat}]\to \bo{Cat}$ preserve flexible limits.
	
	Consider the 2-category $\Z$ generated by the following data below.
	\begin{center}
		\begin{tikzpicture}[baseline=(current  bounding  box.south), scale=2]
			
			\node (0) at (-0.8,0) {$w$};
			\node (a) at (0,0) {$x$};
			\node (b) at (1,0) {$y$};
			\node (c) at (0.5,0) {$\ \ \cong\phi$};
			\node (d) at (1.8,0) {$z$};
			
			\path[font=\scriptsize]
			
			(0) edge[ ->] node [above] {$e$} (a)
			(a) edge[bend left, ->] node [above] {$\pi_1$} (b)
			(a) edge[bend right, ->] node [below] {$\pi_2$} (b)
			(b) edge[->] node [above] {$q$} (d);
			
		\end{tikzpicture}
	\end{center}
	with $q\phi=1$. There is a continuous and cocontinuous 2-functor $T\colon[\Z,\bo{Cat}]\to[\W,\bo{Cat}]$ which acts by sending a diagram $(e,\phi,q)$ to the invertible 2-cell $\phi e$. 
	
	Now let $\P$ be the full subcategory of $[\Z,\bo{Cat}]$ consisting of those diagrams for which $e$ and $q$ are retract equivalences and $\phi$ is the isokernel cell of $q$. We will now see that $T$ restricts to an equivalence $T'\colon\P\to\mt P_M$. First note that if $(e,\phi,q)$ is in $\P$ then $\phi e$ lies in $\mt P_M$ since $\phi$ is acyclic by Proposition~\ref{catexact}. Consider the 2-functor $S\colon \mt P_M\to\P$ defined by sending a 2-cell $\psi$ to the triple $(e,\phi,q)$ where $q$ is the coisoidentifier of $\psi$, $\phi$ is the isokernel cell of $q$, and $e$ is the map induced by $\psi$ into the domain of $\phi$; this is well defined by Remark~\ref{alsoP_M}. It's easy to see that $S$ is an inverse for $T'$, and hence $T'$ is an equivalence of 2-categories.
	
	Now, to prove that the inclusion $ J\colon\mt P_M\hookrightarrow[\W,\bo{Cat}]$ and the restriction of the colimit 2-functor $[\W,\bo{Cat}]\to \bo{Cat}$ to $\mt P_M$ preserve flexible limits, it's enough to show that the inclusion $\P\hookrightarrow [\Z,\bo{Cat}]$ and the 2-functor $[\Z,\bo{Cat}]\to\bo{Cat}$, evaluating at $z$, preserve them. The latter is continuous and cocontinuous (being an evaluation 2-functor) and the former preserves flexible limits since retract equivalences are stable in $\bo{Cat}$ under them (see for example \cite[Section~9]{LR12:articolo}) and isokernel cells are stable under any limits. Thus it follows that $\mt P$ is compatible with flexible limits in $\bo{Cat}$. Moreover the same arguments also show that $\mt P_M$ is closed in $[\W,\bo{Cat}]$ under filtered colimits.
	
	Let's now prove the companion property (II) from Definition~\ref{companion}. Consider a virtually cocomplete 2-category $\A$ with flexible limits, and a small flexible-limit preserving $F\colon \A\to\bo{Cat}$. By virtual cocompleteness of $\A$, the $\V$-category $\P(\A\op)$ is complete; thus for any $X\in\bo{Cat}$ the 2-functor $[X,F-]$ is still small and flexible-limit preserving. Therefore there exists a regular epimorphism
	$$\sum_i(P_i\cdot \A(A_i,-))\twoheadrightarrow [X,F-]$$%
	which, since $[X,F-]$ preserves products and powers, factors through the comparison as a map
	$$\A(\prod_i P_i\cdot A_i,-)\twoheadrightarrow [X,F-]$$%
	which is, in particular, pointwise surjective on objects. The corresponding morphism  $\eta\colon X\to F(\textstyle\prod_i P_i\cdot A_i)$ has the property that any $x\colon X\to FA$ factorizes through $\eta$ via some $\textstyle\prod_i P_i\cdot A_i\to A$; thus $F$ satisfies the solution-set condition.
	
	By \cite[Theorem~6.2]{BLV:articolo} there is a
        pointwise retract equivalence $q\colon \A(A,-)\to F$. Form now
        the isokernel cell $\phi\colon \pi_1\Rightarrow \pi_2\colon
        G\to\A(A,-)$ of $q$. By completeness of $\P(\A\op)$, the
        2-functor $G$ is still small and, since it preserves flexible
        limits, there is once again a pointwise retract equivalence $p\colon \A(B,-)\to G$. It follows that $f:=\pi_1p$, $g:=\pi_2p$, and $\psi:=\phi p$ define a 2-cell that lies pointwise in $\mt P_M$ and which has coisoidentifier equal to $F$. This shows that $\mt P$ is a companion for $\bo{Flex}$.
	
	To conclude, we need only prove that $\mt P_M$ is accessible and accessibly embedded in $[\W,\bo{Cat}]$. By the arguments above, it's enough to show that the 2-category $\P$ is accessible and accessibly embedded in $[\Z,\bo{Cat}]$. Let $\Z'$ be the full subcategory of $\P^\dagger\Z$ spanned by the representables, the isokernel cell $\phi'\colon \pi'_1\Rightarrow \pi'_2$ of the map $e\colon w\to x$, and the equalizers $l$ and $l'$ of $(\pi_1,\pi_2)$ and $(\pi'_1,\pi'_2)$ respectively.
	
	Then define the sketch on $\Z'$ with limit conditions
        $\mathbb{L}$ specifying $\phi$ and $\phi'$ as the isokernel
        cells of $q$ and $e$ respectively, $l$ as both the equalizer
        of $(\pi_1,\pi_2)$ and the identifier of $\phi$, and similarly
        $l'$ as both the equalizer of $(\pi'_1,\pi'_2)$ and the
        identifier of $\phi'$. The colimit cocones $\mathbb{C}$
        specify simply the maps $q$ and $e$ as the coisoidentifiers of
        $\phi$ and $\phi'$ respectively. These conditions say exactly
        that $\phi$ and $\phi'$ are acyclic isokernel cells
        (Remark~\ref{alsoacyclic}), and that $q$ and $e$ are retract
        equivalences (by Proposition~\ref{catexact}); it's
        then easy to see that 
	$$\mt P_M\simeq\P\simeq \tx{Mod}(\W',\mathbb{L},\mathbb{C})$$%
	is therefore accessible. 
\end{proof}

As we did in the previous section, by the presentation of $\mt P_M$
given above and Section~\ref{D-sketch}, the 2-categories of $\mt
P$-models of sketches are the same as the 2-categories of models of
limit sketches where in addition some maps are specified to be retract
equivalences. If we let $\E$ be the class of retract equivalences in
$\bo{Cat}$, we recover the notion of limit/RE sketch. Thus in
this case  Theorems~\ref{relatice-C-charact} and~\ref{C-sketch}
become: 

\begin{teo}\label{flexible-char}
	Let $\A$ be a 2-category; the following are equivalent: \begin{enumerate}\setlength\itemsep{0.25em}
		\item $\A$ is accessible with flexible limits;
		\item $\A$ is accessible and $\mt{P}_1^\dagger \A$ is cocomplete;
		\item $\A$ is accessible and $\mt P$-virtually cocomplete;
		\item $\A$ is accessibly embedded and $\mt P$-virtually reflective in $[\C,\V]$ for some $\C$;
		\item $\A$ is the $\V$-category of models of a
                  limit/\textnormal{RE} sketch.
	\end{enumerate}
\end{teo}

\begin{obs}
	Say that a functor $F\colon\A\to\bo{Cat}$ is RE-weakly representable if there exists $A\in\A$ and a pointwise retract equivalence $q\colon\A(A,-)\to F$. Then an inclusion $J\colon\A\to\K$ into a 2-category $\K$ is $\mt P$-reflective if and only if for any $X\in\K$ the 2-functor $\K(X,J-)$ is RE-weakly represented by some $q\colon\A(A,-)\to \K(X,J-)$ whose isokernel cell is also a RE-weakly representable 2-functor.
\end{obs}

The theorem above gives a characterisation of accessible 2-categories with flexible limits similar to that of \cite[Theorem~9.4]{LR12:articolo} and \cite[Section~9.3]{BLV:articolo}; for the relation between these two characterizations see Theorem~\ref{Flex-Re}. Notice that in \cite{LR12:articolo} and \cite{BLV:articolo} they consider the conical version of accessibility while we deal with the flat one; however these notions coincide by \cite[Theorem~3.15]{LT21:articolo}.

\subsection{Powers}\label{powers}

In this section we denote by $\Lambda$ the class of weights for powers by $\lambda$-small sets.

\begin{Def}\cite[Section~4]{AKV00}
  A functor $H\colon \C\to\bo{Set}$, with small domain $\C$, is called
  {\em $\lambda$-sifted} if the following conditions hold:
  \begin{enumerate}\setlength\itemsep{0.25em}
  \item given less than $\lambda$ elements $x_i\in Hc_i$ there exists an object $c\in\C$ such that each $(c_i,x_i)$ lies in the same component of $\tx{El}(H)$ as some element of $Hc$.
  \item given a set $I$ of cardinality less that $\lambda$, and families $(c,x_i)_{i\in I},(c',x'_i)_{i\in I}$ in $\tx{El}(H)$ such that for each $i$ the pair $(c,x_i),(c',x'_i)$ lies in one connected component of $\tx{El}(H)$, there exists a zig-zag $Z$ in $\C$ connecting $c$ and $c'$ such that each of the pair above can be connected by a zig-zag in $\tx{El}(H)$ whose underlying zig-zag is $Z$.
  \end{enumerate}
  Denote by $\mt{S}^\lambda$ the colimit type given by the
  $\lambda$-sifted functors, with weight $\Delta 1$.
  
  We allow
  $\lambda$ to be $\infty$, in which case there is no restriction on the cardinality of $I$, and we  define $\mt{S}^\infty$ accordingly.
\end{Def}

\begin{prop}\cite[Section~4]{AKV00}
	The colimit of a functor $H\colon\C\to\bo{Set}$ commutes with $\lambda$-small powers if and only if $H$ is $\lambda$-sifted. In particular $\mt{S}^\lambda$ is compatible with $\Lambda$.
\end{prop}

This allows us to prove that $\mt{S}^\lambda$ is a companion for $\Lambda$, and to establish a relationship between being $\lambda$-sifted and being $\Lambda$-precontinuous. 

\begin{prop}
	The colimit type $\mt{S}^\lambda$ is a companion for
        $\Lambda$. Indeed the equality 
	$$ \mt{S}^\lambda_1\A=\Lambda\tx{-PCts}(\A\op,\bo{Set})$$%
       holds for any $\A$ at all. 
\end{prop}
\begin{proof}
  We already know that $\mt{S}^\lambda$ is compatible with
  $\Lambda$, so it will suffice to show that any
  $\Lambda$-precontinuous $F\colon\A\op\to\bo{Set}$ lies in
  $\mt{S}^\lambda_1\A$. Since $F$ is small we can find $H\colon\C\to\A$, with $\C$ small, such that $F\cong \colim YH$. By $\Lambda$-precontinuity of $F$, for each $X\in\bo{Set}_{\lambda}$ and $A\in\A$ we have 
	$$  \colim(X\pitchfork\A(A,H-))\cong X\pitchfork( \colim\A(A,H-)). $$%
	Therefore, by the proposition above, $\A(A,H-)$ is $\lambda$-sifted for any $A\in\A$; in other words $\A(A,H-)\in \mt{S}^\lambda$ for all $A\in\A$. Thus $F\cong\colim YH$ lies in $\mt{S}^\lambda_1\A$.
\end{proof}

\begin{obs}
	If in the proof we take  $H$ to be the projection $\pi_F\colon\tx{El}(F)\to \A$ we obtain the following characterization: a functor $F\colon\A\op\to\bo{Set}$ lies in $ \mt{S}^\lambda_1\A$ if and only if $\tx{El}(F)$ has a small final subcategory (that is, $F$ is small) and $\A(A,\pi_F-)$ is $\lambda$-sifted for any $A\in\A$.
\end{obs}

\begin{obs}
	This can be done more generally for any $\V$ and any class of objects $\G$ in $\V$. Consider the class of weights $\pitchfork_\G$ for powers by elements of $\G$, and let $\mt C_\G$ be the colimit type given by the pairs $( M ,H)$ for which
	$$  M *(G\pitchfork H)\cong G\pitchfork( M *H) $$%
	for any $G\in\G$. 
\end{obs}

\section{Weak reflections}\label{wrsnc}

In this section we plan to capture the standard characterization theorems of \cite{AR94:libro} for accessible categories with products in terms of weak reflection and weak cocompleteness. We also obtain the results of \cite{LR12:articolo} involving accessible 2-categories with flexible limits. To do this we make use of the notion of {\em kernel-quotient system} developed in \cite[Section~2]{BG14:articolo}. 

\begin{Def}
	Let us fix an object $X\in\V$ together with a map $x\colon X\to I$; we define a $\V$-category $\mathbb F$ with three objects $2,1,0$ and homs $\mathbb F(2,2)=\mathbb F(1,1)=\mathbb F(0,0)=I$, $\mathbb F(2,1)=X,\ \mathbb F(2,0)=\mathbb F(1,0)=I$, and $\mathbb F(0,1)=\mathbb F(0,2)=\mathbb F(1,2)=0$; the only non-trivial composition map is $x\colon \mathbb F(1,0)\otimes\mathbb F(2,1)\to\mathbb F(2,0)$. Let now $\mathbb{K}$ be the full subcategory of $\mathbb F$ with objects $2$ and $1$; we depict $\mathbb F$ and $\mathbb K$ as below.
	\begin{center}
		
		\begin{tikzpicture}[baseline=(current  bounding  box.south), scale=2]
			
			\node (b) at (1,0) {$2$};
			\node (c) at (2,0) {$1$};
			\node (d) at (2.8,0) {$0$};
			\node (e) at (1.5,0) {$\mathbb{K}$};

			\path[font=\scriptsize]
			
			([yshift=1.5pt]b.east) edge [->, bend left] node [above] {} ([yshift=1.5pt]c.west)
			([yshift=-1.5pt]b.east) edge [->, bend right] node [below] {} ([yshift=-1.5pt]c.west)
			(c) edge [->>] node [above] {} (d);
		\end{tikzpicture}
	\end{center}
\end{Def}

Let $h\colon\mathbbm{2}\to\mathbb F$ pick out the arrow $1\to 0$ in $\mathbb F$, and $k\colon \mathbb{K}\to\mathbb F$ be the inclusion; then we can consider the adjunction below (as in \cite{BG14:articolo})
\begin{center}
	
	\begin{tikzpicture}[baseline=(current  bounding  box.south), scale=2]

		\node (f) at (0,0.4) {$[\mathbbm{2},\V]$};
		\node (g) at (1.4,0.4) {$[\mathbb{K},\V]$};
		\node (h) at (0.7, 0.45) {$\perp$};
		\path[font=\scriptsize]

		([yshift=-1.4pt]f.east) edge [->] node [below] {$K$} ([yshift=-1.4pt]g.west)
		([yshift=2.1pt]f.east) edge [bend left,<-] node [above] {$Q$} ([yshift=2.1pt]g.west);
	\end{tikzpicture}
	
\end{center}
where $K=k^*\circ \tx{Ran}_h$ and $Q=h^*\circ \tx{Lan}_k$. Given a map $f$ in $\V$ we call $Kf$ the $\mathbb F$-kernel of $f$, and given a diagram $H$ on $\mathbb{K}$ we call the map $QH$ the $\mathbb F$-quotient of $H$.

\begin{obs}
	Note that, for any $H\colon\mathbb{K}\to\V$ and $f\colon A\to H2$ in $\V$, pre-composition with $f$ induces a diagram $H^f\colon\mathbb{K}\to\V$ with $(H^f)1=H1, (H^f)0=H0,$ and $(H^f)2=A$.
\end{obs}

\begin{lema}\label{stabilityquotients}
	Every $\mathbb F$-quotient is an epimorphism in $\V$. Moreover, for any $H\colon\mathbb{K}\to\V$ and any epimorphism $e\colon A\to H2$ in $\V$, we have $QH\cong Q(H^e)$.
\end{lema}
\begin{proof}
	Consider an $\mathbb F$-quotient map $e=QH\colon H1\to P$ and any pair $f,g\colon P\to B$ such that $fe=ge$. Then $fe$ is a cocone for $H$ and factors through $QH$ by $f$ and $g$; by the universal property of the colimit then $f=g$. For the last part of the statement note that, since $e$ is an epimorphism, giving a cocone for $H$ is equivalent to giving a cocone for $H^e$. Therefore $QH\cong Q(H^e)$.
\end{proof}

\begin{as}
	Let $\E$ be a collection of maps in $\V$. From now on we assume that $\mathbb F$ and $\E$ satisfy the following proprieties:
	\begin{enumerate}
		\item Every map in $\E$ is the $\mathbb F$-quotient of its $\mathbb F$-kernel.
		\item $\E$ is closed under composition.
	\end{enumerate} 
\end{as}

If $\E$ consists of all the $\mathbb F$-quotient maps, then (1) is saying that $\mathbb F$-quotient maps are effective in the sense of \cite{BG14:articolo}.

\begin{ese}\label{ese-kernels}The following examples satisfy the conditions above:
	\begin{enumerate}\setlength\itemsep{0.25em}
		
		\item Let $\V$ be either a regular category or have a regular projective unit. Let $\G$ be a dense generator of $\V_0$, and $\mathscr E$ as in Definition~\ref{E-maps} with the properties assumed in Section~\ref{products+G-powers}. Then we can consider the kernel-quotient system for kernel pairs and coequalizers ($X=I+I$ and $x$ is the co-diagonal).
		
		\item Let $\V=\bo{Cat}$; consider $\mathbb F$ generated by $X=\{\cdot\cong\cdot\}$ the free-living isomorphism and $x\colon X\to 1$ the unique map. Then $\mathbb F$-quotients are coisoidentifiers and $\mathbb F$-kernels are isokernel cells. We take $\E$ to consist of the retract equivalences (which are $\mathbb F$-quotients, but not all $\mathbb F$-quotients are retract equivalences).
	\end{enumerate}
\end{ese}

The data of a kernel-quotient system and a class of maps $\E$ induces a colimit type $\mt C$:

\begin{Def}
	Given $\mathbb F$ and $\E$ as above we can define a colimit type $\mt C$ as follows: $\mt C_M$ is non-empty only for $M=\mathbb F(k-,0)\colon\mathbb{K}\op\to\V$ and in that case $\mt C_M$ is the full subcategory of $[\mathbb{K},\V]$ spanned by the diagrams of the form $(Kq)\circ e$ for any compatible $e,q\in\E$.
\end{Def}	

\begin{es}
	In the case of Example~\ref{ese-kernels}(1) the colimit type
        induced is that of $\G$-pseudo equivalence relations
        (Section~\ref{products+G-powers}); in the case of (2) we
        obtain the companion for the class of flexible limits 
        (Section~\ref{flexible}). 
\end{es}

\begin{Def}\cite{LR12:articolo}
	Say that a $\V$-functor $P\colon \A\op\to\V$ is {\em $\E$-weakly representable} if there exists a map $f\colon  YA\to P$ which is pointwise in $\E$. Denote by $\tx{W}_\E(\A)$ the full subcategory of $[\A\op,\V]$ spanned by the $\E$-weakly representables.
\end{Def}

Note that by construction we then have $\mt C_1(\A)\subseteq\tx{W}_\E(\A)$.

\begin{Def}
	We say that a $\V$-functor $F\colon\A\to\B$ is {\em $\E$-weakly reflective} if $\B(B,F-)$ is $\E$-weakly representable for any $B\in\B$.
\end{Def}

When $\V=\bo{Set}$ and $\E$ is the class of surjections this corresponds to the notion of weakly reflective subcategory of \cite{AR94:libro}. More generally, this is part of the framework of \cite{LR12:articolo} where $\E$-weak reflectivity was first considered in the enriched context.

In the following proposition we assume $\mt C_1(\K\op)$ to have $\mathbb F$-kernels of representables; that is true in particular when it has all limits of representables and hence whenever $\K$ is complete. 

\begin{prop}\label{weakl-left}
	The following are equivalent for a fully faithful inclusion $J\colon\A\hookrightarrow\K$ and a  $\V$-category $\K$ for which $\mt C_1(\K\op)$ has $\mathbb F$-kernels of representables:\begin{enumerate}\setlength\itemsep{0.25em}
		\item $\A$ is $\E$-weakly reflective in $\K$;
		\item $\A$ is $\mt C$-virtually reflective in $\K$.
	\end{enumerate}
\end{prop}
\begin{proof} 
	Note that $(1)$ says that each $\K(K,J-)$ is in $\tx{W}_\E(\A\op)$, while $(2)$ says that it is in $\mt C_1(\A\op)$. Then $(2)\Rightarrow(1)$ is trivial since $\mt C_1(\A\op)\subseteq\tx{W}_\E(\A\op)$. 
	
	For $(1)\Rightarrow(2)$ assume that $\A$ is $\E$-weakly reflective in $\K$; we need to prove that $\K(K,J-)$ actually lies in $\mt C_1(\A\op)$. By hypothesis $\K(K,J-)$ is $\E$-weakly representable, so there exists $A\in\A$ together with a map pointwise in $\E$
	$$ q\colon \A(A,-)\twoheadrightarrow \K(K,J-). $$%
	Such a $q$ determines a map $K\to JA$ in $\K$, and this in turn induces a morphism 
	$$ q'\colon\K(JA,-)\longrightarrow\K(K,-)$$%
	which, when restricted to $\A$, gives back $q$. Now, by the hypothesis on $\mt C_1(\K\op)$, the $\mathbb F$-kernel $Kq'$ of $q'$, with domain $S$, lies $\mt C_1(\K\op)$. In particular we obtain a diagram
	\begin{center}
		
		\begin{tikzpicture}[baseline=(current  bounding  box.south), scale=2]
			
			\node (a) at (-0.3,0) {$\K(Q,-)$};
			\node (b) at (0.8,0) {$S$};
			\node (c) at (2.3,0) {$\K(JA,-)$};
			\node (d) at (3.7,0) {$\K(K,-)$};
			\node (e) at (1.4,0) {$Kq'$};

			\path[font=\scriptsize]
			
			(a) edge [->>] node [above] {$s$} (b)
			([yshift=1.5pt]b.east) edge [->, bend left] node [above] {} ([yshift=1.5pt]c.west)
			([yshift=-1.5pt]b.east) edge [->, bend right] node [below] {} ([yshift=-1.5pt]c.west)
			(c) edge [->] node [above] {$q'$} (d);
		\end{tikzpicture}
	\end{center}
	where $s$ is a map pointwise in $\E$, and hence an epimorphism by Lemma~\ref{stabilityquotients}. Now we can restrict this diagram to $\A$ by pre-composing with $J$ and, since pre-composition in continuous, $(Kq')J\cong Kq$ is the $\mathbb F$-kernel of $q$. Moreover $q$ is the $\mathbb F$-quotient of $Kq$ by our initial assumptions on $\E$. Note also that $sJ$ is still pointwise in $\E$ and hence an epimorphism.
	
	By hypothesis $\K(Q,J-)$ is $\E$-weakly representable, so there exists $r\colon\A(B,-)\twoheadrightarrow\K(Q,J-)$ which lies pointwise in $\E$. The map $e=sJ\circ r$ is still pointwise in $\E$ by condition (2); thus we have a presentation as below 
	\begin{center}
		
		\begin{tikzpicture}[baseline=(current  bounding  box.south), scale=2]
			
			\node (a) at (-0.3,0) {$\A(B,-)$};
			\node (b) at (0.8,0) {$SJ$};
			\node (c) at (2.3,0) {$\A(A,-)$};
			\node (d) at (3.7,0) {$\K(K,J-)$};
			\node (e) at (1.45,0) {$Kq$};

			\path[font=\scriptsize]
			
			(a) edge [->>] node [above] {$e$} (b)
			([yshift=1.5pt]b.east) edge [->, bend left] node [above] {} ([yshift=1.5pt]c.west)
			([yshift=-1.5pt]b.east) edge [->, bend right] node [below] {} ([yshift=-1.5pt]c.west)
			(c) edge [->>] node [above] {$q$} (d);
		\end{tikzpicture}
	\end{center}
	showing that $\K(X,J-)$ can be written as an $\mathbb F$-quotient of representables (by condition (1) and Lemma~\ref{stabilityquotients}). Moreover the diagram is constructed so that it lies pointwise in $\mt C$, witnessing that $\K(X,J-)$ lies in $\mt C_1(\A\op)$.
\end{proof}

The notion of {\em $\E$-weak colimit} was also introduced in \cite{LR12:articolo}; we recall its definition below.

\begin{Def}
	Given a $\V$-category $\A$, a weight $M\colon\C\op\to\V$ with small domain, and $H\colon\C\to\A$, we say that the {\em $\E$-weak colimit} of $H$ weighted by $M$ exists in $\A$ if $[\C\op,\V](M,\A(H,-))\colon\A\to\V$ is $\E$-weakly representable. We say that $\A$ is {\em $\E$-weakly cocomplete} if it has all $\E$-weak colimits.
\end{Def}

When $\E$ satisfies the assumption of this section, $\E$-weak cocompleteness and $\mt C$-virtual cocompleteness coincide:

\begin{cor}
	The following are equivalent for a $\V$-category $\A$: \begin{enumerate}\setlength\itemsep{0.25em}
		\item $\A$ is $\E$-weakly cocomplete;
		\item $\A$ is $\mt C$-virtually cocomplete.
	\end{enumerate}
\end{cor}
\begin{proof}
	$(2)\Rightarrow (1)$ is trivial. To prove that $(1)\Rightarrow (2)$ we need to show that if $\tx{W}_\E^\dagger(\A)$ has colimits of representables then so does $\mt C_1^\dagger\A$. Let $\K=\P\A$ and $J\colon\A\hookrightarrow\K$ be the inclusion; then $\K$ is cocomplete and hence satisfies the hypothesis of Proposition~\ref{weakl-left}. Moreover $J$ has an $\E$-weak left adjoint: given $X\in\K$ we can write it as a colimit $X\cong M*JH$ of objects from $\A$, then $\K(X,J-)\cong \{M, YH\}$ is a limit of representables in $[\A,\V]$. In other words $\K(X,J-)$ is a colimit of representables when seen in the opposite $\V$-category, by our assumption then $\K(X,J-)$ lies in $\tx{W}_\E^\dagger(\A)$, as desired. 
	
	It follows by the proposition above that $J$ is $\mt C$-virtually reflective, and thus $\mt C_1^\dagger\A$ has colimits of representables: compute the colimits in $\K$ and then transport them into $\mt C_1^\dagger\A$ through the relative left adjoint. 
\end{proof}

\begin{obs}\label{weak-finite}
	Assume that the unit $I$ and the object $X$ (defining $\mathbb F$) are $\alpha$-presentable; then $\mathbb F$-kernels are $\alpha$-small limits. Thus, if we replace $\P\A$ in the proof above with the free cocompletion under $\alpha$-small colimits, we can prove that $\tx{W}_\E^\dagger(\A)$ has $\alpha$-small colimits of representables if and only if $\mt C_1^\dagger\A$ has them.
\end{obs}

Now we can apply Proposition~\ref{weakl-left} and its corollary in the context of Example~\ref{ese-kernels}(1), where $\E$ is the class of those regular epimorphisms that are stable under $\G$-powers. Then Theorem~\ref{prod+G-powers-weak} becomes:

\begin{teo}\label{weak-prod+pow}
  The following are equivalent for a $\V$-category $\A$:
  \begin{enumerate}\setlength\itemsep{0.25em}
  \item $\A$ is accessible with products and $\G$-powers;
  \item $\A$ is accessible and $\E$-weakly cocomplete;
  \item $\A$ is accessibly embedded and $\E$-weakly reflective in $[\C,\V]$ for some $\C$;
  \item $\A$ is the $\V$-category of models of a limit/$\E$ sketch.
  \end{enumerate}
\end{teo}

In this way we recover the characterization of ordinary accessible categories with products  
given in \cite[Chapter~4]{AR01}; we also obtain an enriched version of it in the context of categories enriched over finitary quasivarieties. 

Similarly, we can apply Proposition~\ref{weakl-left} and its corollary in the context of Example~\ref{ese-kernels}(2), where $\V=\bo{Cat}$ and $\E$ is the class of retract equivalences. Then Theorem~\ref{flexible-char} becomes:

\begin{teo}\label{Flex-Re}
  The following are equivalent for a 2-category $\A$: 
  \begin{enumerate}\setlength\itemsep{0.25em}
  \item $\A$ is accessible with flexible limits;
  \item $\A$ is accessible and $\E$-weakly cocomplete;
  \item $\A$ is accessibly embedded and $\E$-weakly reflective in $[\C,\V]$ for some $\C$;
  \item $\A$ is the $\V$-category of models of a
    limit/\textnormal{RE} sketch.
  \end{enumerate}
\end{teo}

As a consequence we obtain part of \cite[Theorem~9.4]{LR12:articolo} characterizing accessible 2-categories with flexible limits in terms of weak cocompleteness.


\end{document}